\documentclass[12pt,a4paper]{article}
\usepackage{amsmath}
\usepackage{amsfonts}
\usepackage{amsthm}
\usepackage{amssymb}

\usepackage{enumerate}
\theoremstyle{plain}
\newtheorem{theo}{Theorem}[section]
\newtheorem{prop}[theo]{Proposition}
\newtheorem{cor}[theo]{Corollary}%

\theoremstyle{definition}
\newtheorem{definition}[theo]{Definition}
\theoremstyle{remark}
\newtheorem{rem}[theo]{Remark}
\numberwithin{equation}{section}

\newcommand{\C}{\mathbb{C}}

\newcommand{\R}{\mathbb{R}}
\newcommand{\N}{\mathbb{N}}
\newcommand{\M}{\mathbb{M}}
\newcommand{\divrg}{\textrm{div}\,}

\title{A generalized Korn inequality and strong unique continuation for the Reissner-Mindlin plate system
\thanks{The second author is supported by FRA2014 `Problemi inversi per PDE, unicit\`a, stabilit\`a, algoritmi', Universit\`a degli Studi di Trieste, the second and the third author are supported by GNAMPA of the
Istituto Nazionale di Alta Matematica (INdAM)}}
\author{Antonino Morassi\thanks{Dipartimento Politecnico di Ingegneria e Architettura,
Universit\`a degli Studi di Udine, via Cotonificio 114, 33100
Udine, Italy. E-mail: \textsf{antonino.morassi@uniud.it}}, \  Edi
Rosset\thanks{Dipartimento di Matematica e Geoscienze,
Universit\`a degli Studi di Trieste, via Valerio 12/1, 34127
Trieste, Italy. E-mail: \textsf{rossedi@units.it}} \ and Sergio
Vessella\thanks{Dipartimento di Matematica e Informatica ``Ulisse
Dini'', Universit\`a degli Studi di Firenze, Viale Morgagni 67/a,
50134 Firenze, Italy. E-mail: \textsf{sergio.vessella@unifi.it}}}

\date{}

\begin{document}
\maketitle

\begin{abstract}
We prove constructive estimates for elastic plates modelled by the
Reissner-Mindlin theory and made by general anisotropic material.
Namely, we obtain a generalized Korn inequality which allows to
derive quantitative stability and global $H^2$ regularity for the
Neumann problem. Moreover, in case of isotropic material, we
derive an interior three spheres inequality with optimal exponent from which the strong unique continuation property follows.
\end{abstract}


\medskip

\medskip


\section{Introduction}
\label{sec:intro}

In the present paper we consider elastic plates modelled by the
Reissner-Mindlin theory. This theory was developed for moderately
thick plates, that is for plates whose thickness is of the order
of one tenth of the planar dimensions of the middle surface
\cite{Rei45}, \cite{Min51}. Our aim is to give a rigorous,
thorough and self-contained presentation of mathematical results
concerning the Neumann problem, a boundary value problem which
poses interesting features which, at our knowledge, have not yet
been pointed out in the literature.

Throughout the paper we consider an elastic plate $\Omega \times
\left [ - \frac{h}{2}, \frac{h}{2} \right ]$, where $\Omega
\subset \R^2$ is the middle surface and $h$ is the constant
thickness of the plate. A transversal force field $\overline{Q}$
and a couple field $\overline{M}$ are applied at the boundary of
the plate. According to the Reissner-Mindlin model, at any point
$x=(x_1,x_2)$ of $\Omega$ we denote by $w=w(x)$ and
$\omega_\alpha(x)$, $\alpha=1,2$, the infinitesimal transversal
displacement at $x$ and the infinitesimal rigid rotation of the
transversal material fiber thorugh $x$, respectively. Therefore,
the pair $(\varphi,w)$, with
$(\varphi_1=\omega_2,\varphi_2=-\omega_1)$, satisfies the
following Neumann boundary value problem

\begin{center}
\( {\displaystyle \left\{
\begin{array}{lr}
     \mathrm{\divrg}(S(\varphi+\nabla w))=0
      & \mathrm{in}\ \Omega,
        \vspace{0.25em}\\
      \mathrm{\divrg}({\mathbb P}\nabla \varphi)-S(\varphi+\nabla w)=0, & \mathrm{in}\ \Omega,
          \vspace{0.25em}\\
      (S(\varphi+\nabla w))\cdot n= \overline{Q},
      & \mathrm{on}\ \partial \Omega,
        \vspace{0.25em}\\
      ({\mathbb P}\nabla \varphi) n = \overline{M}, &\mathrm{on}\ \partial
      \Omega,
          \vspace{0.25em}\\
\end{array}
\right. } \) \vskip -7.5em
\begin{eqnarray}
& & \label{eq:intro-1}\\
& & \label{eq:intro-2}\\
& & \label{eq:intro-3}\\
& & \label{eq:intro-4}
\end{eqnarray}
\end{center}
where ${\mathbb P}$ and $S$  are the fourth-order bending tensor
and the shearing matrix of the plate, respectively. The vector $n$
denotes the outer unit normal to $\Omega$.

The weak formulation of \eqref{eq:intro-1}--\eqref{eq:intro-4}
consists in determining $(\varphi,w)\in H^1(\Omega, \R^2)\times
H^1(\Omega)$ satisfying
\begin{equation}
  \label{eq:var_for}
   a((\varphi,w), (\psi,v)) =\int_{\partial\Omega}\overline{Q} v + \overline{M}\cdot \psi, \quad  \forall \psi\in H^1(\Omega, \R^2), \forall v\in H^1(\Omega),
\end{equation}
where
\begin{equation}
  \label{eq:bil_for}
    a((\varphi,w), (\psi,v)) = \int_\Omega {\mathbb P}\nabla \varphi\cdot \nabla \psi + \int_\Omega S(\varphi+\nabla w)\cdot
        (\psi+\nabla v).
\end{equation}
The coercivity of the bilinear form $a(\cdot,\cdot)$ in the
subspace
\[{\mathcal H} =\left\{(\psi,v)\in H^1(\Omega,\R^2)\times H^1(\Omega)\ |\ \int_\Omega \psi =0, \int_\Omega v =0\right\}\]
with respect to the norm induced by $H^1(\Omega, \R^2)\times H^1(\Omega)$ is not standard. To prove this property -- in other terms,
the equivalence of the standard norm in ${\mathcal H}$ with the norm induced by the energy functional --
we derive the following generalized Korn-type inequality
\begin{equation}
  \label{eq:gener_korn}
  \|\nabla \varphi\|_{L^2(\Omega)}\leq C\left(\|\widehat{\nabla} \varphi\|_{L^2(\Omega)}+\|\varphi+\nabla w\|_{L^2(\Omega)}\right), \forall \varphi\in H^1(\Omega,\R^2), \forall w\in H^1(\Omega,\R),
\end{equation}
where $\widehat \nabla$ denotes the symmetric part of the gradient
and the constant $C$ is constructively determined in terms of the
parameters describing the geometrical properties of the Lipschitz
domain $\Omega$. Inequality \eqref{eq:gener_korn} allows to solve
the Neumann problem and provides a quantitative stability estimate
in the $H^1$ norm.

Assuming Lipschitz continuous coefficients and $C^{1,1}$
regularity of the boundary, we prove global $H^2$ regularity
estimates. For the proof, which is mainly based on the regularity
theory developed by Agmon \cite{Ag65} and Campanato \cite{Ca80}, a
key role is played by quantitative Poincar\'e inequalities for
functions vanishing on a portion of the boundary, derived in
\cite{A-M-R08}.

Finally, in case of isotropic material, we adapt arguments in
\cite{LNW2010} to $H^2$ solutions of the plate system
\eqref{eq:intro-1}--\eqref{eq:intro-2}, obtaining a three spheres
inequality with optimal exponent and, as a standard consequence, we derive the strong unique continuation property.

Let us notice that the constructive character of all the
estimates derived in the present paper is crucial for possible
applications to inverse problems associated to the Neumann problem
\eqref{eq:intro-1}--\eqref{eq:intro-4}. As a future direction of
research, we plan to use such results to treat inverse problems
concerning the determination of defects, such as elastic
inclusions, in isotropic elastic plates modelled by the
Reissner-Mindlin model.

The paper is organized as follows. In section \ref{sec:notation}
we collect the notation and in section \ref{sec:model} we present
a self-contained derivation of the mechanical model for general
anisotropic material. Section \ref{sec:korn} contains the proof of
the generalized Korn-type inequality \eqref{eq:gener_korn}, which
is the key ingredient used in section \ref{sec:direct} to study
the Neumann problem. In section \ref{sec:reg} we derive $H^2$
global regularity estimates. In Section \ref{sec:UC} we state and
prove the three spheres inequality. Finally, section
\ref{sec:appendix} is an Appendix where we have postponed some
technical estimates about regularity up to the boundary.

\section{Notation} \label{sec:notation}

Let $P=(x_1(P), x_2(P))$ be a point of $\R^2$.
We shall denote by $B_r(P)$ the disk in $\R^2$ of radius $r$ and
center $P$ and by $R_{a,b}(P)$ the rectangle
$R_{a,b}(P)=\{x=(x_1,x_2)\ |\ |x_1-x_1(P)|<a,\ |x_2-x_2(P)|<b \}$. To simplify the notation,
we shall denote $B_r=B_r(O)$, $R_{a,b}=R_{a,b}(O)$.

\begin{definition}
  \label{def:2.1} (${C}^{k,1}$ regularity)
Let $\Omega$ be a bounded domain in ${\R}^{2}$. Given $k\in\N$, we say that a portion $S$ of
$\partial \Omega$ is of \textit{class ${C}^{k,1}$ with
constants $\rho_{0}$, $M_{0}>0$}, if, for any $P \in S$, there
exists a rigid transformation of coordinates under which we have
$P=0$ and
\begin{equation*}
  \Omega \cap R_{\frac{\rho_0}{M_0},\rho_0}=\{x=(x_1,x_2) \in R_{\frac{\rho_0}{M_0},\rho_0}\quad | \quad
x_{2}>\psi(x_1)
  \},
\end{equation*}
where $\psi$ is a ${C}^{k,1}$ function on
$\left(-\frac{\rho_0}{M_0},\frac{\rho_0}{M_0}\right)$ satisfying
\begin{equation*}
\psi(0)=0, \quad \psi' (0)=0, \quad \hbox {when } k \geq 1,
\end{equation*}
\begin{equation*}
\|\psi\|_{{C}^{k,1}\left(-\frac{\rho_0}{M_0},\frac{\rho_0}{M_0}\right)} \leq M_{0}\rho_{0}.
\end{equation*}

\medskip
\noindent When $k=0$ we also say that $S$ is of
\textit{Lipschitz class with constants $\rho_{0}$, $M_{0}$}.
\end{definition}
\begin{rem}
  \label{rem:2.1}
  We use the convention to normalize all norms in such a way that their
  terms are dimensionally homogeneous with the $L^\infty$ norm and coincide with the
  standard definition when the dimensional parameter equals one, see \cite{MRV07} for details.

\end{rem}

Given a bounded domain $\Omega$ in $\R^2$ such that $\partial
\Omega$ is of class $C^{k,1}$, with $k\geq 1$, we consider as
positive the orientation of the boundary induced by the outer unit
normal $n$ in the following sense. Given a point
$P\in\partial\Omega$, let us denote by $\tau=\tau(P)$ the unit
tangent at the boundary in $P$ obtained by applying to $n$ a
counterclockwise rotation of angle $\frac{\pi}{2}$, that is
$\tau=e_3 \times n$,
where $\times$ denotes the vector product in $\R^3$, $\{e_1,
e_2\}$ is the canonical basis in $\R^2$ and $e_3=e_1 \times e_2$.
%
%
%

We denote by $\mathbb{M}^2$ the space of $2 \times 2$ real valued
matrices and by ${\mathcal L} (X, Y)$ the space of bounded linear
operators between Banach spaces $X$ and $Y$.

For every $2 \times 2$ matrices $A$, $B$ and for every $\mathbb{L}
\in{\mathcal L} ({\mathbb{M}}^{2}, {\mathbb{M}}^{2})$, we use the
following notation:
\begin{equation}
  \label{eq:2.notation_1}
  ({\mathbb{L}}A)_{ij} = L_{ijkl}A_{kl},
\end{equation}
\begin{equation}
  \label{eq:2.notation_2}
  A \cdot B = A_{ij}B_{ij}, \quad |A|= (A \cdot A)^{\frac {1} {2}}.
\end{equation}
Notice that here and in the sequel summation over repeated indexes
is implied.
%

\section{The Reissner-Mindlin plate model} \label{sec:model}

The Reissner-Mindlin plate is a classical model for plates having
moderate thickness \cite{Rei45}, \cite{Min51}. The
Reissner-Mindlin plate theory can be rigorously deduced {}from the
three-dimensional linear elasticity using arguments of
$\Gamma$-convergence of the energy functional, as it was shown in
\cite{P-PPG-T-07}. Our aim in this section is more modest, namely,
we simply derive the boundary value problem governing the statical
equilibrium of an elastic Reissner-Mindlin plate under Neumann
boundary conditions following the engineering approach of the
Theory of Structures. This allows us to introduce some notations
useful in the sequel and to make the presentation of the physical
problem complete.

Let us consider a plate $\Omega \times \left [ - \frac{h}{2},
\frac{h}{2} \right ]$ with middle surface represented by a bounded
domain $\Omega$ in $\R^2$ having uniform thickness $h$ and
boundary $\partial \Omega$ of class $C^{1,1}$. In this section we
adopt the convention that Greek indexes assume the values $1,2$,
whereas Latin indexes run {}from $1$ to $3$.

We follow the direct approach to define the infinitesimal
deformation of the plate. In particular, we restrict ourselves to
the case in which the points $x=(x_1,x_2)$ of the middle surface
$\Omega$ are subject to transversal displacement $w(x_1,x_2)e_3$,
and any transversal material fiber $\{x\}\times \left [ -
\frac{h}{2}, \frac{h}{2} \right ]$, $x\in \Omega$, undergoes an
infinitesimal rigid rotation $\omega(x)$, with $\omega(x)\cdot e_3
=0$. In this section we shall be concerned exclusively with
regular functions on their domain of definition. The above
kinematical assumptions imply that the displacement field present
in the plate is given by the following three-dimensional vector
field:
\begin{equation}
  \label{eq:anto-1.2}
  u(x,x_3)=w(x)e_3 + x_3 \varphi (x), \quad x\in
  \overline{\Omega}, \ |x_3| \leq \frac{h}{2},
\end{equation}
where
\begin{equation}
  \label{eq:anto-1.3}
  \varphi (x) = \omega (x) \times e_3, \quad x\in
  \overline{\Omega}.
\end{equation}
By \eqref{eq:anto-1.2} and \eqref{eq:anto-1.3}, the associated
infinitesimal strain tensor $E[u]\in \M^3$ takes the form
\begin{equation}
  \label{eq:anto-2.3}
  E[u](x,x_3) \equiv (\nabla u)^{sym}(x,x_3)= x_3 (\nabla_x
  \varphi(x)  )^{sym} + (\gamma(x) \otimes e_3)^{sym},
\end{equation}
where $\nabla_x (\cdot)=  \frac{\partial}{\partial x_\alpha}
(\cdot ) e_\alpha $ is the surface gradient operator,
$\nabla^{sym}(\cdot)= \frac{1}{2} ( \nabla (\cdot) + \nabla^T
(\cdot))$, and
\begin{equation}
  \label{eq:anto-2.4}
  \gamma(x)=\varphi(x) + \nabla_x w(x).
\end{equation}
Within the approximation of the theory of infinitesimal
deformations, $\gamma$ expresses the angular deviation between the
transversal material fiber at $x$ and the normal direction to the
deformed middle surface of the plate at $x$.

The classical deduction of the mechanical model of a thin plate
follows essentially {}from integration over the thickness of the
corresponding three-dimensional quantities. In particular, taking
advantage of the infinitesimal deformation assumption, we can
refer the independent variables to the initial undeformed
configuration of the plate.

Let us introduce an arbitrary portion $\Omega' \times \left [ -
\frac{h}{2}, \frac{h}{2} \right ]$ of plate, where $\Omega'
\subset \subset \Omega$ is a subdomain of $\Omega$ with regular
boundary. Consider the material fiber $\{x\}\times \left [ -
\frac{h}{2}, \frac{h}{2} \right ]$ for $x\in \partial \Omega'$ and
denote by $t(x,x_3,e_\alpha)\in \R^3$, $|x_3| \leq \frac{h}{2}$,
the \textit{traction vector} acting on a plane containing the
direction of the fiber and orthogonal to the direction $e_\alpha$.
By Cauchy's Lemma we have $t(x,x_3,e_\alpha)=T(x,x_3)e_\alpha$,
where $T(x,x_3) \in \M^{3}$ is the (symmetric) Cauchy stress
tensor at the point $(x,x_3)$. Denote by $n$ the unit outer normal
vector to $\partial \Omega'$ such that $n\cdot e_3=0$. To simplify
the notation, it is convenient to consider $n$ as a
two-dimensional vector belonging to the plane $x_3=0$ containing
the middle surface $\Omega$ of the plate. By the classical Stress
Principle for plates, we postulate that the two complementary
parts $\Omega'$ and $\Omega \setminus \Omega'$ interact with one
another through a field of force vectors $R=R(x,n)\in \R^3$ and
couple vectors $M=M(x,n) \in \R^3$ assigned per unit length at $x
\in
\partial \Omega'$. Denoting by
\begin{equation}
  \label{eq:anto-4.1}
  R(x,e_\alpha) = \int_{-h/2}^{h/2} t(x,x_3, e_\alpha) dx_3
\end{equation}
the force vector (per unit length) acting on a direction
orthogonal to $e_\alpha$ and passing through $x \in
\partial \Omega'$, the contact force $R(x,n)$ can be expressed as
\begin{equation}
  \label{eq:anto-4.2}
  R(x,n) = T^\Omega (x) n, \quad x \in \partial \Omega',
\end{equation}
where the \textit{surface force tensor} $T^\Omega (x) \in
\M^{3\times 2}$ is given by
\begin{equation}
  \label{eq:anto-4.3}
  T^\Omega (x) = R(x,e_\alpha) \otimes e_\alpha, \quad \hbox{in }
  \Omega.
\end{equation}
Let $P=I-e_3 \otimes e_3$ be the projection of $\R^3$ along the
direction $e_3$. $T^{\Omega}$ is decomposed additively by $P$ in
its \textit{membranal} and \textit{shearing} component
\begin{equation}
  \label{eq:anto-4.4}
  T^\Omega  = PT^\Omega + (I-P)T^\Omega \equiv T^{\Omega(m)}+
  T^{\Omega(s)},
\end{equation}
where, following the standard nomenclature in plate theory, the
components $T_{\alpha \beta}^{\Omega(m)}$ ($=T_{\beta
\alpha}^{\Omega(m)}$), $\alpha, \beta =1,2$, are called the
\textit{membrane forces} and the components $T_{3
\beta}^{\Omega(s)}$, $\beta=1,2$, are the \textit{shear forces}
(also denoted as $T_{3 \beta}^{\Omega(s)} = Q_\beta$). The
assumption of infinitesimal deformations and the hypothesis of
vanishing in-plane displacements of the middle surface of the
plate allow us to take
\begin{equation}
  \label{eq:anto-5.1}
  T^{\Omega(m)} = 0, \quad \hbox{in } \Omega.
\end{equation}
Denote by
\begin{equation}
  \label{eq:anto-5.2}
  M(x,e_\alpha) = \int_{-h/2}^{h/2} x_3 e_3 \times t (x,x_3,e_\alpha)
  dx_3, \quad \alpha=1,2,
\end{equation}
the contact couple acting at $x \in \partial \Omega'$ on a
direction orthogonal to $e_\alpha$ passing through $x$. Note that
$M(x,e_\alpha) \cdot e_3=0$ by definition, that is $M(x,e_\alpha)$
actually is a two-dimensional couple field belonging to the middle
plane of the plate. Analogously to \eqref{eq:anto-4.2}, we have
\begin{equation}
  \label{eq:anto-5.3}
  M(x,n) = M^\Omega (x) n, \quad x \in \partial \Omega',
\end{equation}
where the \textit{surface couple tensor} $M^\Omega (x) \in
\M^{3\times 2}$ has the expression
\begin{equation}
  \label{eq:anto-surface_couple_tensor}
  M^\Omega (x) = M(x,e_\alpha) \otimes e_\alpha.
\end{equation}
A direct calculation shows that
\begin{equation}
  \label{eq:anto-6.1}
  M(x,e_\alpha) = e_3 \times e_\beta M_{\beta \alpha}(x),
\end{equation}
where
\begin{equation}
  \label{eq:anto-6.2}
  M_{\beta \alpha}(x)=\int_{-h/2}^{h/2}x_3 T_{\beta
  \alpha}(x,x_3)dx_3, \quad \alpha, \beta =1,2,
\end{equation}
are the \textit{bending moments} (for $\alpha=\beta$) and the
\textit{twisting moments} (for $\alpha \neq \beta$) of the plate
at $x$ (per unit length).

Denote by $q(x)e_3$ the external transversal force per unit area
acting in $\Omega$. The statical equilibrium of the plate is
satisfied if and only if the following two equations are
simultaneously satisfied:
\begin{center}
\( {\displaystyle \left\{
\begin{array}{lr}
    \int_{\partial \Omega'} T^{\Omega}n ds + \int_{\Omega'}qe_3
    dx=0,
        \vspace{0.25em}\\
    \int_{\partial \Omega'} \left ( (x-x_0) \times T^{\Omega}n
    + M^{\Omega}n \right ) ds + \int_{\Omega'} (x-x_0)\times qe_3 dx=0,
          \vspace{0.25em}\\
\end{array}
\right. } \) \vskip -4.5em
\begin{eqnarray}
& & \label{eq:anto-7.1}\\
& & \label{eq:anto-7.2}
\end{eqnarray}
\end{center}
for every subdomain $\Omega' \subseteq \Omega$, where $x_0$ is a
fixed point. By applying the Divergence Theorem in $\Omega'$ and
by the arbitrariness of $\Omega'$ we deduce
\begin{center}
\( {\displaystyle \left\{
\begin{array}{lr}
    {\rm div}_x T^{\Omega(s)} + qe_3=0,
    & \mathrm{in}\ \Omega,
          \vspace{0.25em}\\
    {\rm div}_x M^{\Omega} + (T^{\Omega(s)})^Te_3 \times e_3=0,
    & \mathrm{in}\ \Omega.
          \vspace{0.25em}\\
\end{array}
\right. } \) \vskip -4.5em
\begin{eqnarray}
& & \label{eq:anto-7.3}\\
& & \label{eq:anto-7.4}
\end{eqnarray}
\end{center}
Consider the case in which the boundary of the plate $\partial
\Omega$ is subjected simultaneously to a couple field
$\overline{M}^*$, $\overline{M}^*\cdot e_3=0$, and a transversal
force field $\overline{Q}e_3$. Local equilibrium considerations on
points of $\partial \Omega$ yield the following boundary
conditions:
\begin{center}
\( {\displaystyle \left\{
\begin{array}{lr}
    M^{\Omega}n = \overline{M}^*,
    & \mathrm{on}\ \partial \Omega,
          \vspace{0.25em}\\
    T^{\Omega(s)}n= \overline{Q}e_3,
    & \mathrm{on}\ \partial \Omega.
          \vspace{0.25em}\\
\end{array}
\right. } \) \vskip -4.5em
\begin{eqnarray}
& & \label{eq:anto-7.6}\\
& & \label{eq:anto-7.7}
\end{eqnarray}
\end{center}
where $n$ is the unit outer normal to $\partial \Omega$. In
cartesian components, the equilibrium equations
\eqref{eq:anto-7.3}--\eqref{eq:anto-7.7} take the form
\begin{center}
\( {\displaystyle \left\{
\begin{array}{lr}
     M_{\alpha\beta, \beta} - Q_\alpha=0,
      & \mathrm{in}\ \Omega, \mathrm{ \alpha=1,2},
        \vspace{0.25em}\\
      Q_{\alpha,\alpha} + q=0, & \mathrm{in}\ \Omega,
          \vspace{0.25em}\\
      M_{\alpha\beta} n_\beta = \overline{M}_\alpha,
      & \mathrm{on}\ \partial \Omega,
        \vspace{0.25em}\\
      Q_\alpha n_\alpha = \overline{Q}, &\mathrm{on}\ \partial
      \Omega,
          \vspace{0.25em}\\
\end{array}
\right. } \) \vskip -7.5em
\begin{eqnarray}
& & \label{eq:anto-8.1}\\
& & \label{eq:anto-8.2}\\
& & \label{eq:anto-8.3}\\
& & \label{eq:anto-8.4}
\end{eqnarray}
\end{center}
where we have defined $\overline{M}_1= \overline{M}^*_2$ and
$\overline{M}_2= - \overline{M}^*_1$.

To complete the formulation of the equilibrium problem, we need to
introduce the constitutive equation of the material. We limit
ourselves to the Reissner-Mindlin theory and we choose to regard
the kinematical assumptions $E_{33}[u]=0$ as internal constraint,
that is we restrict the possible deformations of the points of the
plate to those whose infinitesimal strain tensor belongs to the
set
\begin{equation}
  \label{eq:anto-8.6}
  {\cal{M}}=\{E \in \M^{3\times3}|  E=E^T,  E \cdot A=0,
  \hbox{ for } A= e_3 \otimes e_3 \}.
\end{equation}
Therefore, by the \textit{Generalized Principle of Determinism}
\cite{Truesdell}, the Cauchy stress tensor $T$ at any point
$(x,x_3)$ of the plate is additively decomposed in an
\textit{active} (symmetric) part $T_{A}$ and in a
\textit{reactive} (symmetric) part $T_{R}$:
\begin{equation}
  \label{eq:anto-9.1}
  T=T_{A}+T_{R},
\end{equation}
where $T_{R}$ does not work in any admissible motion, e.g., $T_{R}
\in {\cal{M}}^\perp$. Consistently with the Principle, the active
stress $T_{A}$ belongs to ${\cal{M}}$ and, in cartesian
coordinates, we have
\begin{equation}
  \label{eq:anto-9.2}
  T_{A} = T_{A \alpha\beta} e_\alpha \otimes e_\beta +
    T_{A \alpha 3} e_\alpha \otimes e_3 + T_{A 3\alpha} e_3 \otimes
    e_\alpha, \quad
  \alpha, \beta=1,2,
\end{equation}
\begin{equation}
  \label{eq:anto-9.3}
  T_{R} = T_{R 33} e_3 \otimes e_3.
\end{equation}
In linear theory, on assuming the reference configuration
unstressed, the active stress in a point $(x,x_3)$ of the plate,
$x \in \overline{\Omega}$ and $|x_3|\leq h/2$, is given by a
linear mapping {}from ${\cal{M}}$ into itself by means of the
fourth order \textit{elasticity tensor} $\C_{\cal{M}} \in
{\mathcal L} ({\mathbb{M}}^{3}, {\mathbb{M}}^{3})$:
\begin{equation}
  \label{eq:anto-9.4}
  T_{A} = \C_{\cal{M}} E[u].
\end{equation}
We assume that $\C_{\cal{M}}$ is constant over the thickness of
the plate and satisfies the minor and major symmetry conditions
expressed in cartesian coordinates as (we drop the subscript
${\cal{M}}$)
\begin{equation}
  \label{eq:anto-9.5}
  C_{ijrs}=C_{jirs}= C_{ijsr}=C_{rsij}, \quad i,j,r,s= 1,2,3, \quad \hbox{in } \Omega.
\end{equation}
Using \eqref{eq:anto-9.1} and recalling \eqref{eq:anto-5.1}, we
have:
\begin{equation}
  \label{eq:anto-10.1}
    T^{\Omega}=T^{\Omega(s)}_A, \quad  M^{\Omega}=M^{\Omega}_A,
\end{equation}
that is, both the shear forces and the moments have active nature.
By \eqref{eq:anto-9.4}, after integration over the thickness, the
surface force tensor and the surface couple tensor are given by
\begin{equation}
  \label{eq:anto-10.1bis}
    T^{\Omega}(x) = h \C(x) ( \gamma \otimes e_3 )^{sym}, \quad \hbox{in } \Omega,
\end{equation}
\begin{equation}
  \label{eq:anto-10.2}
    M^{\Omega}(x) = \frac{h^3}{12}  {\cal{E}}\C(x) (\nabla_x
    \varphi(x))^{sym}, \quad \hbox{in } \Omega,
\end{equation}
where ${\cal{E}} \in \M^3$ is the unique skew-symmetric matrix
such that ${\cal{E}} a = e_3 \times a$ for every $a \in \R^3$. The
constitutive equations \eqref{eq:anto-10.1bis},
\eqref{eq:anto-10.2} can be written in more expressive way in
terms of the cartesian components of shear forces and
bending-twisting moments, namely
\begin{equation}
  \label{eq:anto-10.3bis}
    Q_\alpha= S_{\alpha \beta}(x) (\varphi_\beta + w,_\beta), \quad \alpha=1,2,
\end{equation}
\begin{equation}
  \label{eq:anto-10.3}
    M_{\alpha\beta}= P_{\alpha \beta \gamma \delta}(x)
    \varphi_{\gamma, \delta}, \quad \alpha, \beta=1,2.
\end{equation}
where the \textit{plate shearing matrix} $S \in {\mathbb{M}}^{2}$
and the \textit{plate bending tensor} $\mathbb P \in {\mathcal L}
({\mathbb{M}}^{2}, {\mathbb{M}}^{2})$ are given by
\begin{equation}
  \label{eq:anto-10.4bis}
    S_{\alpha \beta}(x) =h C_{3 \alpha 3 \beta}(x), \quad \alpha,
    \beta=1,2,
\end{equation}
\begin{equation}
  \label{eq:anto-10.4}
    P_{\alpha \beta \gamma \delta} (x) =\frac{h^3}{12} C_{\alpha \beta \gamma \delta}(x), \quad \alpha, \beta, \gamma,
    \delta=1,2.
\end{equation}
{}From the symmetry assumptions \eqref{eq:anto-9.5} on the elastic
tensor $\C$ it follows that the shearing matrix $S$ is symmetric
and the bending tensor $\mathbb P$ satisfies the minor and major
symmetry conditions, namely (in cartesian coordinates)
\begin{equation}
  \label{eq:anto-11.1}
    S_{\alpha \beta} = S_{\beta \alpha}, \quad \alpha, \beta =1,2,
    \ \ \hbox{in } \Omega,
\end{equation}
\begin{equation}
  \label{eq:anto-11.2}
    P_{\alpha \beta \gamma \delta}= P_{\beta \alpha \gamma \delta}= P_{\alpha \beta \delta \gamma}=P_{\gamma \delta \alpha \beta}, \quad \alpha, \beta, \gamma,
    \delta=1,2, \ \ \hbox{in } \Omega.
\end{equation}
We recall that the symmetry conditions \eqref{eq:anto-11.2} are
equivalent to
\begin{equation}
  \label{eq:anto-11.3}
    {\mathbb P} A = {\mathbb P} \widehat{A}, \quad {\mathbb P} A \ \hbox{is symmetric}, \quad
    {\mathbb P} A \cdot B = {\mathbb P} B \cdot A,
\end{equation}
for every $2 \times 2$ matrices $A$, $B$, where, here and in the sequel, we denote for brevity $\widehat{A}=A^{sym}$.

On $S$ and $\mathbb P$ we also make the following assumptions.
\begin{enumerate}[I)]
\item \textit{Regularity (boundedness)}
\begin{equation}
  \label{eq:anto-11.4}
  S \in L^\infty(\Omega,   {\cal L} ({\M}^{2})),
\end{equation}
\begin{equation}
  \label{eq:anto-11.5}
  \mathbb P \in L^\infty(\Omega,   {\cal L} ({\M}^{2}, {\M}^{2})).
\end{equation}
\item \textit{Ellipticity (strong convexity)} There exist two
positive constants $\sigma_0$, $\sigma_1$ such that
\begin{equation}
  \label{eq:convex-S}
    h \sigma_0 |v|^2 \leq Sv \cdot v \leq h \sigma_1 |v|^2, \qquad \hbox{a.e. in } \Omega,
\end{equation}
for every $v \in \R^2$, and there exist two positive constants
$\xi_0$, $\xi_1$ such that
\begin{equation}
  \label{eq:convex-P}
     \frac{h^3}{12} \xi_0 | \widehat{A} |^2 \leq {\mathbb P}A \cdot A \leq \frac{h^3}{12} \xi_1 | \widehat{A} |^2, \qquad \hbox{a.e. in } \Omega,
\end{equation}
for every $2\times 2$ matrix $A$.
\end{enumerate}
Finally, under the above notation and in view of
\eqref{eq:anto-10.3bis}--\eqref{eq:anto-10.3}, the problem
\eqref{eq:anto-8.1}--\eqref{eq:anto-8.4} for $q \equiv 0$ in
$\Omega$ takes the form \eqref{eq:intro-1}--\eqref{eq:intro-4},
namely (in cartesian components)
\begin{center}
\( {\displaystyle \left\{
\begin{array}{lr}
     (P_{\alpha\beta\gamma\delta} \varphi_{\gamma, \delta}),_\beta - S_{\alpha \beta}(\varphi_\beta +w,_\beta)=0,
      & \mathrm{in}\ \Omega,
        \vspace{0.25em}\\
      (S_{\alpha \beta}(\varphi_\beta +w,_\beta)),_\alpha=0, & \mathrm{in}\ \Omega,
          \vspace{0.25em}\\
      (P_{\alpha\beta\gamma\delta} \varphi_{\gamma, \delta}) n_\beta = \overline{M}_\alpha,
      & \mathrm{on}\ \partial \Omega,
        \vspace{0.25em}\\
      S_{\alpha \beta}(\varphi_\beta +w,_\beta) n_\alpha = \overline{Q}, &\mathrm{on}\ \partial
      \Omega.
          \vspace{0.25em}\\
\end{array}
\right. } \) \vskip -7.5em
\begin{eqnarray}
& & \label{eq:anto-12.1}\\
& & \label{eq:anto-12.2}\\
& & \label{eq:anto-12.3}\\
& & \label{eq:anto-12.4}
\end{eqnarray}
\end{center}

\section{A generalized Korn inequality} \label{sec:korn}

Throughout this section, $\Omega$ will be a bounded domain in $\R^2$, with boundary
of Lipschitz class with  constants $\rho_0$, $M_0$, satisfying
\begin{equation}
  \label{eq:korn1}
    \hbox{diam}(\Omega)\leq M_1\rho_0,
\end{equation}
\begin{equation}
  \label{eq:korn2}
    B_{s_0\rho_0}(x_0)\subset\Omega,
\end{equation}
for some $s_0>0$ and $x_0\in \Omega$.
For any $E\subset\Omega$, we shall denote by

\begin{equation}
  \label{eq:korn3}
    x_E=\frac{1}{|E|}\int_E x,
\end{equation}

\begin{equation}
  \label{eq:korn4}
    v_E=\frac{1}{|E|}\int_E v,
\end{equation}
the center of mass of $E$ and the integral mean of a function $v$
with values in $\R^n$, $n\geq 1$, respectively.

In order to prove the generalized Korn inequality of Theorem
\ref{theo:Korn_gener2}, let us recall the constructive Poincar\'e
and classical Korn inequalities.

\begin{prop} [Poincar\'e inequalities]
   \label{prop:Poinc}
There exists a positive constant $C_P$ only depending on $M_0$ and $M_1$, such that
for every $u\in H^1(\Omega,\R^n)$, $n=1,2$,
\begin{equation}
  \label{eq:korn5}
    \|u-u_\Omega\|_{L^2(\Omega)}\leq C_P\rho_0\|\nabla u\|_{L^2(\Omega)},
\end{equation}

\begin{equation}
  \label{eq:korn6}
    \|u-u_E\|_{H^1(\Omega)}\leq
        \left(1+\left(\frac{|\Omega|}{|E|}\right)^\frac{1}{2}\right)\sqrt{1+C_P^2}\ \rho_0\|\nabla u\|_{L^2(\Omega)}.
\end{equation}
\end{prop}

See for instance \cite[Example 3.5]{A-M-R08} and also \cite{A-M-R02} for a quantitative evaluation of the constant $C_P$.

\begin{prop} [Korn inequalities]
   \label{prop:Korn_classica}
There exists a positive constant $C_K$ only depending on $M_0$ and $M_1$, such that
for every $u\in H^1(\Omega,\R^2)$,
\begin{equation}
  \label{eq:korn7}
    \left\|\nabla u-\frac{1}{2}(\nabla u - \nabla^T u)_\Omega\right\|_{L^2(\Omega)}\leq C_K\|\widehat{\nabla} u\|_{L^2(\Omega)},
\end{equation}

\begin{equation}
  \label{eq:korn8}
    \left\|u-u_E-\frac{1}{2}(\nabla u - \nabla^T u)_E(x-x_E)\right\|_{H^1(\Omega)}\leq C_{E,\Omega}C_K
        \sqrt{1+C_P^2}\ \rho_0\|\widehat{\nabla} u\|_{L^2(\Omega)},
\end{equation}
where
\begin{equation}
  \label{eq:korn9}
    C_{E,\Omega} = 1+\left(2\frac{|\Omega|}{|E|}\left(1+M_1^2\right)\right)^\frac{1}{2}.
\end{equation}
\end{prop}

See the fundamental paper by Friedrichs \cite{F47} on second Korn inequality and also
\cite[Example 5.3]{A-M-R08} for a proof of \eqref{eq:korn8}-\eqref{eq:korn9}.

Notice that, when $E=B_{s_0\rho_0}(x_0)$, $C_{E,\Omega} \leq 1+\sqrt 2\left(1+M_1^2\right)^
\frac{1}{2}\frac{M_1}{s_0}$.

The following generalized Korn-type inequality is useful for the
study of the Reissner-Mindlin plate system.
\begin{theo} [Generalized second Korn inequality]
   \label{theo:Korn_gener2}
There exists a positive constant $C$ only depending on $M_0$, $M_1$ and $s_0$, such that,
for every $\varphi\in H^1(\Omega,\R^2)$ and for every $w\in H^1(\Omega,\R)$,

\begin{equation}
  \label{eq:korn17}
    \|\nabla \varphi\|_{L^2(\Omega)}\leq C\left(\|\widehat{\nabla} \varphi\|_{L^2(\Omega)}+\frac{1}{\rho_0}\|\varphi+\nabla w\|_{L^2(\Omega)}\right).
\end{equation}
\end{theo}

\begin{proof}
We may assume, with no loss of generality, that $\int_{B_{s_0\rho_0}(x_0)}\varphi =0$.
Let
\begin{equation*}
    \mathcal{S}=\left\{\nabla w\ |\ w\in H^1(\Omega), \int_\Omega w =0\right\}
        \subset L^2(\Omega, \R^2).
\end{equation*}
$\mathcal{S}$ is a closed subspace of $L^2(\Omega, \R^2)$. In fact, let
$\nabla w_n\in \mathcal{S}$ and $F\in L^2(\Omega, \R^2)$ such that $\nabla w_n \rightarrow F$ in $L^2(\Omega, \R^2)$. By the Poincar\'e inequality \eqref{eq:korn5}, $w_n$ is a Cauchy sequence in $H^1(\Omega)$, so that there exists
$w\in H^1(\Omega)$ such that $w_n \rightarrow w$ in $H^1(\Omega)$. Therefore $F=\nabla w\in \mathcal{S}$. By the projection theorem, for every
$\varphi \in L^2(\Omega, \R^2)$, there exists a unique $\nabla \overline{w}\in \mathcal{S}$ such that
\begin{equation}
   \label{eq:korn17bis}
    \|\varphi -\nabla \overline{w}\|_{L^2(\Omega)} =
        \min_{\nabla w\in S}\|\varphi -\nabla w\|_{L^2(\Omega)} =
        \min_{\nabla w\in S}\|\varphi +\nabla w\|_{L^2(\Omega)}.
\end{equation}
Moreover, $\nabla \overline{w}$ is characterized by the condition
\begin{equation}
  \label{eq:korn18}
    \varphi -\nabla \overline{w} \perp \nabla w\ \hbox{in } L^2(\Omega), \hbox{ for every }\nabla w \in S.
\end{equation}

Let us consider the infinitesimal rigid displacement

\begin{equation}
\label{eq:korn18bis}
    r=\frac{1}{2}(\nabla \varphi - \nabla^T \varphi)_{B_{s_0\rho_0}(x_0)}(x-x_0) := W(x-x_0),
\end{equation}
where
\begin{equation*}
    W=
        \left(
 \begin{array}{cc}
   0&\alpha\\
   -\alpha & 0
  \end{array}
\right) \ ,
\end{equation*}
that is
\begin{equation*}
    r=(\alpha(x-x_0)_2,-\alpha(x-x_0)_1).
\end{equation*}

Let us distinguish two cases:

i) $\Omega=B_{s_0\rho_0}(x_0)$,

ii) $B_{s_0\rho_0}(x_0)\subsetneq\Omega$.

\medskip

\emph{Case i).}
Let us see that, when one takes $\varphi = r$ in \eqref{eq:korn17bis}, with $r$ given by \eqref{eq:korn18bis}, then its projection into $\mathcal{S}$ is
\begin{equation}
\label{eq:korn18ter}
    \nabla \overline{w} =0,
\end{equation}
that is, by the equivalent condition \eqref{eq:korn18},
$r\perp \nabla w$ in $L^2(\Omega)$, for every $\nabla w \in \mathcal{S}$. In fact

\begin{multline}
  \label{eq:korn19}
    \int_{B_{s_0\rho_0}(x_0)} r\cdot \nabla w = \int_{B_{s_0\rho_0}(x_0)}
        \alpha(x-x_0)_2w_{x_1}-\alpha(x-x_0)_1w_{x_2} =\\
        =\alpha
         \int_{\partial B_{s_0\rho_0}(x_0)}w\left((x-x_0)_2\nu_1-(x-x_0)_1\nu_2\right).
\end{multline}

Since $\nu = \frac{x-x_0}{s_0\rho_0}$, we have
\begin{equation*}
(x-x_0)_2\nu_1-(x-x_0)_1\nu_2=(x-x_0)_2\frac{(x-x_0)_1}{s_0\rho_0}-
(x-x_0)_1\frac{(x-x_0)_2}{s_0\rho_0}=0,
\end{equation*}
so that
\begin{equation}
  \label{eq:korn19bis}
    \int_{B_{s_0\rho_0}(x_0)} r\cdot \nabla w = 0, \qquad \hbox{for every } \nabla w \in S.
\end{equation}

Therefore, by \eqref{eq:korn17bis} and \eqref{eq:korn18ter},
\begin{equation}
  \label{eq:korn20}
    \|r\|_{L^2(\Omega)}\leq \|r+\nabla w\|_{L^2(\Omega)}, \qquad\hbox{for every }\nabla w \in S.
\end{equation}
By the definition of $r$ and recalling that $\Omega = B_{s_0\rho_0}(x_0)$, it follows trivially that
\begin{equation}
  \label{eq:korn21}
    \|r\|_{L^2(\Omega)}^2 =\frac{\pi}{2}\alpha^2s_0^4\rho_0^2 = \frac{s_0^2\rho_0^2}{4}
         \|\nabla r\|_{L^2(\Omega)}^2.
\end{equation}

By the Korn inequality \eqref{eq:korn8}, by \eqref{eq:korn20} and \eqref{eq:korn21}, we have
\begin{multline}
  \label{eq:korn22}
    \|\nabla \varphi\|_{L^2(\Omega)}\leq \|\nabla (\varphi-r)\|_{L^2(\Omega)} +
        \|\nabla r\|_{L^2(\Omega)}= \\
        = \|\nabla (\varphi-r)\|_{L^2(\Omega)}+
        \frac{2}{s_0\rho_0}\|r\|_{L^2(\Omega)}
        \leq \|\nabla (\varphi-r)\|_{L^2(\Omega)}+
        \frac{2}{s_0\rho_0}\|r +\nabla w\|_{L^2(\Omega)}\leq\\
        \leq
        \|\nabla (\varphi-r)\|_{L^2(\Omega)}+
        \frac{2}{s_0\rho_0}\|\varphi +\nabla w\|_{L^2(\Omega)}+\frac{2}{s_0\rho_0}
        \|\varphi -r\|_{L^2(\Omega)} \leq\\
        \leq C\left(\|\widehat{\nabla} \varphi\|_{L^2(\Omega)}+
        \frac{1}{\rho_0}\|\varphi +\nabla w\|_{L^2(\Omega)}\right),
\end{multline}
with $C$ only depending on $M_0$, $M_1$ and $s_0$.

\medskip

\emph{Case ii).}

Let $r$ be the infinitesimal rigid displacement given by
\eqref{eq:korn18bis}. By \eqref{eq:korn18}, its projection $\nabla
\overline{w}$ into $\mathcal{S}$ satisfies

\begin{equation}
  \label{eq:korn23}
    \int_\Omega r\cdot \nabla w = \int_\Omega \nabla \overline{w}\cdot \nabla w,
        \hbox{ for every }\nabla w \in S.
\end{equation}
Choosing, in particular, $w=\overline{w}$ in \eqref{eq:korn23}, and by the same arguments used to prove
\eqref{eq:korn19bis}, we have
\begin{equation}
  \label{eq:korn24}
    \int_\Omega |\nabla \overline{w}|^2 = \int_\Omega r\cdot\nabla \overline{w}
        =\int_{B_{\frac{s_0\rho_0}{2}}(x_0)}r\cdot\nabla \overline{w}+\int_{\Omega\setminus B_{\frac{s_0\rho_0}{2}}(x_0)}r\cdot\nabla \overline{w}
        =\int_{\Omega\setminus B_{\frac{s_0\rho_0}{2}}(x_0)}r\cdot\nabla \overline{w},
\end{equation}
so that, by H\"{o}lder inequality,
\begin{equation}
  \label{eq:korn25}
    \|\nabla \overline{w}\|_{L^2(\Omega)}\leq  \| r\|_{L^2(\Omega\setminus B_{\frac{s_0\rho_0}{2}}(x_0))}.
\end{equation}
By a direct computation, we have
\begin{equation}
  \label{eq:korn26}
    \frac{\int_{\Omega\setminus B_{\frac{s_0\rho_0}{2}}(x_0)}|r|^2}{\int_\Omega |r|^2}
        = 1 - \frac{\int_{B_{\frac{s_0\rho_0}{2}}(x_0)}|r|^2}{\int_\Omega |r|^2}\leq
        1- \frac{\int_{B_{\frac{s_0\rho_0}{2}}(x_0)}|r|^2}{\int_{B_{s_0\rho_0}(x_0)} |r|^2}=\frac{15}{16},
\end{equation}
and, by \eqref{eq:korn25} and \eqref{eq:korn26},
\begin{equation}
  \label{eq:korn27}
    \|\nabla \overline{w}\|_{L^2(\Omega)}\leq  \frac{\sqrt{15}}{4}\| r\|_{L^2(\Omega)}.
\end{equation}
Therefore
\begin{equation}
  \label{eq:korn28}
    \|r-\nabla \overline{w}\|_{L^2(\Omega)}\geq
         \|r\|_{L^2(\Omega)}-  \|\nabla \overline{w}\|_{L^2(\Omega)}\geq
        \left(1- \frac{\sqrt {15}}{4}\right) \|r\|_{L^2(\Omega)}.
\end{equation}
{}From \eqref{eq:korn17bis} and \eqref{eq:korn28}, it follows that
\begin{equation}
  \label{eq:korn29}
         \|r\|_{L^2(\Omega)}\leq \frac{4}{4-\sqrt{15}} \|r+\nabla w\|_{L^2(\Omega)},\qquad
        \hbox{for every } w\in H^1(\Omega).
\end{equation}

Now, $\nabla r = W$, $|\nabla r|^2 =2\alpha^2$, so that
\begin{equation}
   \label{eq:korn29bis}
    \int_\Omega |\nabla r|^2 \leq 8\alpha^2\pi M_1^2\rho_0^2.
\end{equation}

Since
$|W(x-x_0)|^2=\alpha^2|x-x_0|^2$,
by\eqref{eq:korn29bis}, we have
\begin{equation}
   \label{eq:korn29ter}
    \int_\Omega|r|^2=\alpha^2\int_\Omega|x-x_0|^2\geq
      \frac{\pi}{2}\alpha^2s_0^4\rho_0^4
        \geq\left(\frac{s_0^2}{4M_1}\right)^2\rho_0^2\int_\Omega|\nabla r|^2.
\end{equation}

By \eqref{eq:korn8},  \eqref{eq:korn29} and \eqref{eq:korn29ter},
\begin{multline}
  \label{eq:korn30}
    \|\nabla \varphi\|_{L^2(\Omega)}\leq \|\nabla (\varphi-r)\|_{L^2(\Omega)} +
        \|\nabla r\|_{L^2(\Omega)}\leq \\
        \leq C\left(\|\nabla (\varphi-r)\|_{L^2(\Omega)}+
        \frac{1}{\rho_0}\|r\|_{L^2(\Omega)}\right)
        \leq C\left(\|\nabla (\varphi-r)\|_{L^2(\Omega)}+
        \frac{1}{\rho_0}\|r +\nabla w\|_{L^2(\Omega)}\right)\leq\\
        \leq
        C\left(\|\nabla (\varphi-r)\|_{L^2(\Omega)}+
        \frac{1}{\rho_0}\|\varphi +\nabla w\|_{L^2(\Omega)}+\frac{1}{\rho_0}
        \|\varphi -r\|_{L^2(\Omega)} \right)\leq\\
        \leq C\left(\|\widehat{\nabla} \varphi\|_{L^2(\Omega)}+
        \frac{1}{\rho_0}\|\varphi +\nabla w\|_{L^2(\Omega)}\right),
\end{multline}
with $C$ only depending on $M_0$, $M_1$ and $s_0$.

Notice that a more accurate estimate can be obtained by replacing $B_{\frac{s_0\rho_0}{2}}(x_0)$ with $B_{s_0\rho_0}(x_0)$ in \eqref{eq:korn24} and in what follows, obtaining
\begin{equation}
  \label{eq:korn31}
    \|\nabla \overline{w}\|_{L^2(\Omega)}\leq  \sqrt \gamma\| r\|_{L^2(\Omega)},
\end{equation}
where the constant $\gamma$,
\begin{equation}
  \label{eq:korn32}
    \gamma = \frac{\int_{\Omega\setminus B_{s_0\rho_0}(x_0)}|r|^2}{\int_\Omega |r|^2}
        = 1 - \frac{\frac{\pi}{2}s_0^4\rho_0^4}{\int_\Omega |x-x_0|^2}<1
\end{equation}
can be easily estimated in terms of the geometry of $\Omega$.
\end{proof}

\begin{rem}
  \label{rem:Gobert}
Let us notice that, choosing in particular $w\equiv 0$ in \eqref{eq:korn17}, it follows that there exists a positive constant $C$ only depending on $M_0$, $M_1$ and $s_0$, such that
for every $u\in H^1(\Omega,\R^2)$,

\begin{equation}
  \label{eq:Gobert}
    \|u\|_{H^1(\Omega)}\leq C(\rho_0\|\widehat{\nabla} u\|_{L^2(\Omega)}+\|u\|_{L^2(\Omega)}).
\end{equation}
The above inequality was first proved by Gobert in \cite{G62} by using the theory of singular integrals, a different proof for regular domains being presented by Duvaut and Lions in \cite{DL76}.
\end{rem}

\section{The Neumann problem} \label{sec:direct}
Let us consider a plate $\Omega \times \left [ - \frac{h}{2},
\frac{h}{2} \right ]$ with middle surface represented by a bounded
domain $\Omega$ in $\R^2$ having uniform thickness $h$, subject to
a transversal force field $\overline{Q}$ and to a couple field
$\overline{M}$ acting on its boundary. Under the kinematic
assumptions of Reissner-Mindlin's theory, the pair $(\varphi,w)$,
with $\varphi = (\varphi_1,\varphi_2)$, where $\varphi_\alpha$,
$\alpha=1,2$, are expressed in terms of the infinitesimal rigid
rotation field $\omega$ by \eqref{eq:anto-1.3} and $w$ is the
transversal displacement, satisfies the equilibrium problem
\eqref{eq:intro-1}-\eqref{eq:intro-4}. The shearing matrix $S \in
L^\infty(\Omega,   {\cal L} ({\M}^{2}))$ and the bending tensor
${\mathbb P}\in L^\infty(\Omega,   {\cal L} ({\M}^{2},
{\M}^{2}))$, introduced in Section \ref{sec:model}, are assumed to
satisfy the symmetry conditions \eqref{eq:anto-11.1},
\eqref{eq:anto-11.2} and the ellipticity conditions
\eqref{eq:convex-S}, \eqref{eq:convex-P}, respectively.

Summing up the weak formulation of equations \eqref{eq:intro-1} and \eqref{eq:intro-2}, one derives the following \textit{weak formulation} of the equilibrium problem \eqref{eq:intro-1}-\eqref{eq:intro-4}:

\medskip
\textit{A pair} $(\varphi,w)\in H^1(\Omega, \R^2)\times
H^1(\Omega)$ \textit{is a weak solution to} \eqref{eq:intro-1}-\eqref{eq:intro-4} \textit{if}
\textit{for every} $\psi\in H^1(\Omega, \R^2)$ \textit{and for every} $v\in H^1(\Omega)$,
\begin{equation}
  \label{eq:dir1}
    \int_\Omega {\mathbb P}\nabla \varphi\cdot \nabla \psi + \int_\Omega S(\varphi+\nabla w)\cdot
        (\psi+\nabla v)=\int_{\partial\Omega}\overline{Q} v + \overline{M}\cdot \psi.
\end{equation}
Choosing $\psi\equiv 0$, $v\equiv 1$,  in \eqref{eq:dir1}, we have
\begin{equation}
  \label{eq:dir2}
    \int_{\partial\Omega}\overline{Q}=0.
\end{equation}
Inserting $\psi\equiv -b$, $v=b\cdot x$ in \eqref{eq:dir1}, we have
\begin{equation*}
    \int_{\partial\Omega}b\cdot(\overline{Q}x-\overline{M})=0, \qquad \hbox{for every }b\in \R^2,
\end{equation*}
so that
\begin{equation}
  \label{eq:dir3}
    \int_{\partial\Omega}\overline{Q}x-\overline{M}=0.
\end{equation}
We refer to \eqref{eq:dir2}-\eqref{eq:dir3} as the \textit{compatibility conditions} for the equilibrium problem.
\begin{rem}
    \label{rem:all_sol}
Given a solution $(\varphi,w)$  to the equilibrium problem \eqref{eq:intro-1}-\eqref{eq:intro-4}, then all its solutions are given by
\begin{equation}
  \label{eq:dir4}
    \varphi^* = \varphi-b,\quad w^* = w +b\cdot x + a,\qquad \forall a\in \R, \forall b\in \R^2.
\end{equation}
It is obvious that any $(\varphi^*,w^*)$ given by   \eqref{eq:dir4} is a solution. Viceversa, given two solutions $(\varphi,w)$, $(\varphi^*,w^*)$, by subtracting their weak formulations one has
\begin{multline*}
    \int_\Omega {\mathbb P}\nabla (\varphi-\varphi^*)\cdot \nabla \psi + \int_\Omega S((\varphi-\varphi^*)+\nabla (w-w^*))\cdot
        (\psi+\nabla v)=0,\\
        \quad \forall v\in H^1(\Omega),
        \forall \psi\in H^1(\Omega, \R^2).
\end{multline*}
Choosing $\psi=\varphi-\varphi^*$, $v=w-w^*$, and by the ellipticity conditions
\eqref{eq:convex-S}, \eqref{eq:convex-P}, we have
\begin{multline}
    \label{eq:dir5}
    0=\int_\Omega {\mathbb P}\nabla (\varphi-\varphi^*)\cdot \nabla (\varphi-\varphi^*) + \int_\Omega S((\varphi-\varphi^*)+\nabla (w-w^*))\cdot
        ((\varphi-\varphi^*)+\nabla (w-w^*))\geq\\
        \geq\frac{h^3}{12}\xi_0\int_\Omega |\widehat{\nabla}  (\varphi-\varphi^*)|^2+h\sigma_0
        \int_\Omega |(\varphi-\varphi^*)+\nabla (w-w^*)|^2.
\end{multline}
{}From the generalized Korn inequality \eqref{eq:korn17} it follows that
$\nabla (\varphi-\varphi^*)=0$, so that there exists $b\in\R^2$ such that
$\varphi^* = \varphi-b$. By the above inequality we also have that $\nabla (w^*-w)=
\varphi-\varphi^* = b$, and therefore there exists $a\in\R$ such that
$w^* = w +b\cdot x + a$.

An alternative proof of $\nabla (\varphi-\varphi^*)=0$, that
better enlightens the mathematical aspects of the Reissner-Mindlin
model, is based on a qualitative argument which avoids the use of
\eqref{eq:korn17}. Precisely, {}from   \eqref{eq:dir5}, one has
that $\widehat{\nabla}  (\varphi-\varphi^*)=0$ and $\nabla
(w-w^*)= \varphi^*-\varphi$. Therefore $\varphi-\varphi^*=Wx+b$
for some skew symmetric matrix $W=
        \left(
 \begin{array}{cc}
   0&\alpha\\
   -\alpha & 0
  \end{array}
\right) $ and some constant $b\in\R^2$ and $\nabla (w^*-w)=Wx+b\in
\C^\infty$. Hence we can compute $(w^*-w)_{x_1x_2}=(\alpha
x_2+b_1)_{x_2} =\alpha$, $(w^*-w)_{x_2x_1}=(-\alpha x_1+b_2)_{x_1}
=-\alpha$ and, by the Schwarz theorem, $\alpha=0$, so that
$\varphi-\varphi^*=b$.
\end{rem}
\begin{prop}
   \label{prop:diretto}
  Let $\Omega$ be a bounded domain in $\R^2$ with boundary
of Lipschitz class with  constants $\rho_0$, $M_0$, satisfying
\eqref{eq:korn1}-\eqref{eq:korn2}. Let the second order tensor $S
\in L^\infty(\Omega,   {\cal L} ({\M}^{2}))$ and the forth order
tensor ${\mathbb P}\in L^\infty(\Omega,   {\cal L} ({\M}^{2},
{\M}^{2}))$ satisfy the symmetry conditions \eqref{eq:anto-11.1},
\eqref{eq:anto-11.2} and the ellipticity conditions
\eqref{eq:convex-S}, \eqref{eq:convex-P}, respectively. Let
$\overline{M}\in H^{-\frac{1}{2}}(\partial\Omega, \R^2)$ and
$\overline{Q}\in H^{-\frac{1}{2}}(\partial\Omega)$ satisfy the
compatibility conditions \eqref{eq:dir2}-\eqref{eq:dir3}
respectively. Problem \eqref{eq:intro-1}-\eqref{eq:intro-4} admits
a unique solution $(\varphi,w)\in H^1(\Omega,\R^2)\times
H^1(\Omega)$ normalized by the conditions
\begin{equation}
  \label{eq:dir6}
    \int_\Omega \varphi=0, \qquad \int_\Omega w =0.
\end{equation}
Moreover
\begin{equation}
  \label{eq:dir7}
    \|\varphi\|_{H^1(\Omega)} + \frac{1}{\rho_0}\|w\|_{H^1(\Omega)}
        \leq C\left(\|\overline{M}\|_{H^{-\frac{1}{2}}(\partial\Omega)}+\rho_0
        \|\overline{Q}\|_{H^{-\frac{1}{2}}(\partial\Omega)}\right),
\end{equation}
with $C$ only depending on $M_0$, $M_1$, $s_0$, $\xi_0$, $\sigma_0$, $\frac{\rho_0}{h}$.
\end{prop}
\begin{proof}
Let us consider the linear space
\begin{equation}
  \label{eq:dir8}
    {\mathcal H} =\left\{(\psi,v)\in H^1(\Omega,\R^2)\times H^1(\Omega)\ |\ \int_\Omega \psi =0, \int_\Omega v =0\right\},
\end{equation}
which is a Banach space equipped with the norm
\begin{equation}
  \label{eq:dir9}
   \|(\psi,v)\|_{\mathcal H} = \|\psi\|_ {H^1(\Omega)} + \frac{1}{\rho_0}\|v\|_{H^1(\Omega)}.
\end{equation}
The symmetric bilinear form
\begin{equation*}
    a: {\mathcal H}\times {\mathcal H}\rightarrow \R
\end{equation*}
\begin{equation}
   \label{eq:dir9bis}
    a((\varphi,w),(\psi,v))=\int_\Omega {\mathbb P}\nabla\varphi\cdot\nabla \psi +S(\varphi+\nabla w)\cdot(\psi+\nabla v),
\end{equation}
is continuous in ${\mathcal H}\times {\mathcal H}$. Let us see that it is also coercive. By the
ellipticity conditions \eqref{eq:convex-S}, \eqref{eq:convex-P},
\begin{multline}
  \label{eq:dir10}
   a((\varphi,w),(\varphi,w))\geq \frac{h^3}{12}\xi_0 \int_\Omega |\widehat{\nabla }\varphi|^2
    + h\sigma_0 \int_\Omega |\varphi +\nabla w|^2\geq \\
    \geq h^3\min\left\{\frac{\xi_0}{12},\sigma_0\left(\frac{\rho_0}{h}\right)^2\right\}\left(\int_\Omega |\widehat{\nabla }\varphi|^2
    + \frac{1}{\rho_0^2}\int_\Omega |\varphi +\nabla w|^2\right).
\end{multline}
On the other hand, {}from Poincar\'e and Korn inequalities \eqref{eq:korn5} and \eqref{eq:korn17}, and by the trivial estimate $\|\nabla w\|_{L^2(\Omega)}\leq \|\varphi + \nabla w\|_{L^2(\Omega)}
+\|\varphi\|_{L^2(\Omega)}$, one has
\begin{equation}
  \label{eq:dir11}
   \|(\varphi,w)\|_{\mathcal H}\leq C\left(\rho_0\|\widehat{\nabla} \varphi\|_{L^2(\Omega)} +
    \|\varphi + \nabla w\|_{L^2(\Omega)}\right),
\end{equation}
with $C$ only depending on $M_0$, $M_1$ and $s_0$.

{}From \eqref{eq:dir10}-\eqref{eq:dir11}, one has
\begin{equation}
  \label{eq:dir12}
      a((\varphi,w),(\varphi,w))\geq C\|(\varphi,w)\|_{\mathcal H}^2,
\end{equation}
where $C$ only depends on $M_0$, $M_1$, $s_0$, $\xi_0$, $\sigma_0$, $\frac{\rho_0}{h}$.

Therefore the bilinear form \eqref{eq:dir9bis} is a scalar product
inducing an equivalent norm in ${\mathcal H}$, which we denote by $\lvert\lvert\lvert\cdot \rvert\rvert\rvert$.

The linear functional
\begin{equation*}
    F: {\mathcal H}\rightarrow \R
\end{equation*}
\begin{equation*}
    F(\psi,v)=\int_{\partial\Omega} {\widehat Q} v+ {\widehat M}\cdot \psi
\end{equation*}
is bounded and, by \eqref{eq:dir12}, it satisfies
\begin{multline}
  \label{eq:dir13}
      |F(\psi,v)|\leq C\left(\|\overline{M}\|_{H^{-\frac{1}{2}}(\partial\Omega)}\|\psi\|_
        {H^{\frac{1}{2}}(\partial\Omega)}+
        \|\overline{Q}\|_{H^{-\frac{1}{2}}(\partial\Omega)}\|v\|_
        {H^{\frac{1}{2}}(\partial\Omega)}\right)\leq\\
        \leq C\left(\|\overline{M}\|_{H^{-\frac{1}{2}}(\partial\Omega)}+\rho_0
        \|\overline{Q}\|_{H^{-\frac{1}{2}}(\partial\Omega)}\right)\|(\psi,v)\|_{\mathcal H}\leq\\
        \leq
        C\left(\|\overline{M}\|_{H^{-\frac{1}{2}}(\partial\Omega)}+\rho_0
        \|\overline{Q}\|_{H^{-\frac{1}{2}}(\partial\Omega)}\right)\lvert\lvert\lvert(\psi,v)
        \rvert\rvert\rvert,
\end{multline}
so that
\begin{equation}
  \label{eq:dir14}
      \lvert\lvert\lvert F
        \rvert\rvert\rvert_*\leq C\left(\|\overline{M}\|_{H^{-\frac{1}{2}}(\partial\Omega)}+\rho_0
        \|\overline{Q}\|_{H^{-\frac{1}{2}}(\partial\Omega)}\right),
\end{equation}
with $C$ only depending on $M_0$, $M_1$, $s_0$, $\xi_0$, $\sigma_0$, $\frac{\rho_0}{h}$.
By the Riesz representation theorem, there exists a unique $(\varphi,w)\in {\mathcal H}$
such that $a((\varphi,w),(\psi,v)) = F(\psi,v)$ for every $(\psi,v)\in {\mathcal H}$, that is
\eqref{eq:dir1} holds for every $(\psi,v)\in {\mathcal H}$. Moreover
\begin{equation}
  \label{eq:dir15}
      \lvert\lvert\lvert (\varphi,w)
        \rvert\rvert\rvert = \lvert\lvert\lvert F
        \rvert\rvert\rvert_*.
\end{equation}
Let us prove \eqref{eq:dir1} for every $(\psi,v)\in  H^1(\Omega,\R^2)\times H^1(\Omega)$.
Given any $\psi \in  H^1(\Omega,\R^2)$ and any $v\in H^1(\Omega)$, let
\begin{equation*}
      \widetilde{\psi} = \psi-\psi_\Omega, \qquad \widetilde{v} = v+\psi_\Omega\cdot(x-x_\Omega)
        -v_\Omega.
\end{equation*}
We have that $\widetilde{\psi}+ \nabla \widetilde{v} =\psi+\nabla v$.
Hence, by the compatibility conditions \eqref{eq:dir2}-\eqref{eq:dir3},
\begin{multline}
  \label{eq:dir16}
      \int_\Omega {\mathbb P}\nabla \varphi\cdot\nabla \psi + S(\varphi+\nabla w)\cdot(\psi+\nabla v) =
        \int_{\partial\Omega}\overline{M}\cdot\widetilde{\psi}+\overline{Q}\widetilde{v}=\\
        =\int_{\partial\Omega}\overline{M}\cdot\psi+\overline{Q} v -
        \psi_\Omega\cdot\int_{\partial\Omega}(\overline{M}-\overline{Q}x)-v_\Omega
        \int_{\partial\Omega}\overline{Q} - \psi_\Omega\cdot x_\Omega\int_{\partial\Omega}\overline{Q}
        = \int_{\partial\Omega}\overline{M}\cdot\psi+\overline{Q} v.
\end{multline}
Finally, \eqref{eq:dir7} follows {}from \eqref{eq:dir12}, \eqref{eq:dir14} and \eqref{eq:dir15}.
\end{proof}

\section{$H^2$ regularity} \label{sec:reg}

Our main result is the following global regularity theorem.

\begin{theo}
   \label{theo:global-reg}

Let $\Omega$ be a bounded domain in $\R^2$ with boundary of class
$C^{1,1}$, with constants $\rho_0$, $M_0$, satisfying
\eqref{eq:korn1}, \eqref{eq:korn2}. Let $S \in
C^{0,1}(\overline{\Omega}, {\cal L} ({\M}^{2}))$ and $\mathbb P
\in C^{0,1} (\overline{\Omega}, {\cal L} ({\M}^{2}, {\M}^{2}))$
satisfy the symmetry conditions \eqref{eq:anto-11.1},
\eqref{eq:anto-11.2} and the ellipticity conditions
\eqref{eq:convex-S}, \eqref{eq:convex-P}. Let $\overline{M} \in
H^{ \frac{1}{2}} (\partial \Omega, \R^2)$ and $\overline{Q} \in
H^{ \frac{1}{2}} (\partial \Omega)$ satisfy the compatibility
conditions \eqref{eq:dir2}, \eqref{eq:dir3}, respectively. Then,
the weak solution $(\varphi, w) \in H^1(\Omega, \R^2) \times
H^1(\Omega)$ of the problem
\eqref{eq:intro-1}--\eqref{eq:intro-4}, normalized by the
conditions \eqref{eq:dir6}, is such that $(\varphi, w) \in
H^2(\Omega, \R^2) \times H^2(\Omega)$ and
\begin{equation}
  \label{eq:reg-1-1}
      \|\varphi\|_{H^2(\Omega)} + \frac{1}{\rho_0}
      \|w\|_{H^2(\Omega)} \leq C \left (
      \|\overline{M}\|_{H^{\frac{1}{2}} (\partial \Omega)} + \rho_0
      \|\overline{Q}\|_{H^{\frac{1}{2}} (\partial \Omega)} \right
      ),
\end{equation}
where the constant $C>0$ only depends on $M_0$, $M_1$, $s_0$,
$\xi_0$, $\sigma_0$, $ \frac{\rho_0}{h}$, $\|S\|_{
C^{0,1}(\overline{\Omega})}$ and $\|\mathbb P \|_{
C^{0,1}(\overline{\Omega})}$.
\end{theo}
\medskip

The proof of the theorem is mainly based on the approach to
regularity for second order elliptic systems adopted, for
instance, in \cite{Ag65} and \cite{Ca80}. For the sake of
completeness, the main steps of the proof are recalled in the
sequel.

Let us introduce the following notation. Let
\begin{equation}
  \label{eq:reg-2-1}
      B_\sigma^+ = \{(y_1,y_2)\in \R^2| \ y_1^2+y_2^2 < \sigma^2,
      \ y_2 >0\}
\end{equation}
be the hemidisk of radius $\sigma$, $\sigma >0$, and let
\begin{equation}
  \label{eq:reg-2-2}
      \Gamma_\sigma = \{(y_1,y_2)\in \R^2| \ -\sigma \leq y_1 \leq \sigma, \ y_2 =0\}
\end{equation}
and
\begin{equation}
  \label{eq:reg-2-3}
      \Gamma_\sigma^+ = \partial B_\sigma^+ \setminus
      \Gamma_\sigma
\end{equation}
be the flat and the curved portion of the boundary $\partial
B_\sigma^+$, respectively. Moreover, let
\begin{equation}
  \label{eq:reg-2-4}
      H_{\Gamma_\sigma^+}^1 (B_\sigma^+) = \{ g \in
      H^1(B_\sigma^+)| \ g=0 \ \hbox{on } \Gamma_\sigma^+\}.
\end{equation}

Without loss of generality, hereinafter we will assume $\rho_0=1$.
Moreover, the dependence of the constants $C$ on the plate
thickness $h$ will be not explicitly indicated in the estimates
below.

By the regularity of $\partial \Omega$, we can construct a finite
collection of open sets $\Omega_0$, $\{ \Omega_j \}_{j=1}^N$ such
that $\Omega = \Omega_0 \cup \left ( \bigcup_{j=1}^N {\cal
T}_{(j)}^{-1}(B_{ \frac{\sigma}{2}}^+) \right )$, $\Omega_0
\subset \Omega_{\delta_0}$, where $\Omega_{\delta_0}=\{x\in\Omega
\ | \ dist(x, \partial \Omega) > \delta_0\}$, $\delta_0>0$ only
depends on $M_0$, and $N$ only depends on $M_0$, $M_1$. Here,
${\cal T}_{(j)}$, $j=1,...,N$, is a homeomorphism of $C^{1,1}$
class which maps $\Omega_j$ into $B_\sigma$, $\Omega_j \cap
\Omega$ into $B_\sigma^+$, $\overline{\Omega}_j \cap \partial
\Omega$ into $\Gamma_\sigma$, and $\partial \Omega_j \cap \Omega$
into $\Gamma_\sigma^+$.

The estimate of $(\|\varphi\|_{H^2(\Omega_0)} +
\|w\|_{H^2(\Omega_0)})$ is a consequence of the following local
interior regularity result, whose proof can be obtained, for
example, by adapting the arguments illustrated in \cite{Ca80}.

\begin{theo}
\label{theo:reg-loc-inter}
Let us denote by $B_\sigma$ the open ball in $\R^2$ centered at
the origin and with radius $\sigma$, $\sigma >0$. Let $(\varphi,w)
\in H^1(B_\sigma,\R^2)\times H^1(B_\sigma)$ be such that
\begin{equation}
  \label{eq:reg-3-3}
      a((\varphi,w), (\psi,v))=0, \quad \hbox{for every }
      (\psi,v)\in H^1(B_\sigma,\R^2)\times H^1(B_\sigma),
\end{equation}
where
\begin{equation}
  \label{eq:reg-4-1}
      a((\varphi,w), (\psi,v))= \int_{B_\sigma} \mathbb P \nabla
      \varphi \cdot \nabla \psi + \int_{B_\sigma} S(\varphi
      +\nabla w)\cdot (\psi + \nabla v),
\end{equation}
with $\mathbb P \in C^{0,1}(\overline{B}_\sigma, {\cal L}
({\M}^2,{\M}^2))$, $S \in C^{0,1}(\overline{B}_\sigma, {\cal L}
({\M}^2))$ satisfying the symmetry conditions
\eqref{eq:anto-11.1}, \eqref{eq:anto-11.2} and the ellipticity
conditions \eqref{eq:convex-S}, \eqref{eq:convex-P}. Then,
$(\varphi,w) \in H^2(B_\sigma, \R^2) \times H^2(B_\sigma)$ and we
have
\begin{equation}
  \label{eq:reg-4-2}
      \|\varphi\|_{H^2( B_{  \frac{\sigma}{2}}  )} +
      \|w\|_{H^2( B_{  \frac{\sigma}{2}}  )} \leq C
      \left (
      \|\varphi\|_{H^1(B_\sigma)} + \|w\|_{H^1(B_\sigma)}
      \right ),
\end{equation}
where the constant $C>0$ only depends on $\xi_0$, $\sigma_0$,
$\|S\|_{ C^{0,1}(\overline{B}_\sigma)}$ and $\|\mathbb P \|_{
C^{0,1}(\overline{B}_\sigma)}$.
\end{theo}
In order to complete the proof of the regularity estimate, let us
control the quantity $(\|\varphi\|_{H^2(\Omega_j \cap \Omega)} +
\|w\|_{H^2(\Omega_j \cap \Omega)})$ for every $j \in \{1,...,N\}$.

For every $v \in H^1_{\partial \Omega_j \cap \Omega}(\Omega_j \cap
\Omega)$ and for every $\psi \in H^1_{\partial \Omega_j \cap
\Omega}(\Omega_j \cap \Omega, \R^2)$, the solution $(\varphi, w)$
to \eqref{eq:intro-1}--\eqref{eq:intro-4} satisfies the weak
formulation
\begin{multline}
  \label{eq:reg-4-3}
      \int_{\Omega_j \cap \Omega} \mathbb P(x) \nabla_x \varphi \cdot
      \nabla_x \psi dx + \int_{\Omega_j \cap \Omega} S(x) (\varphi
      + \nabla_x w) \cdot (\psi +\nabla_x v) dx =
      \\
      = \int_{\Omega_j
      \cap \partial \Omega} (\overline{Q}v + \overline{M}\cdot \psi)
      ds_x.
\end{multline}
Let us introduce the change of variables
\begin{equation}
  \label{eq:reg-5-1}
      y={\cal{T}}_{(j)}(x), \quad y \in B_\sigma^+,
\end{equation}
\begin{equation}
  \label{eq:reg-5-2}
      x={\cal{T}}_{(j)}^{-1}(y), \quad x \in \Omega_j \cap \Omega,
\end{equation}
and let us define
\begin{equation}
  \label{eq:reg-5-3}
      \widetilde{w}(y)=w({\cal{T}}_{(j)}^{-1}(y)), \quad
      \widetilde{\varphi}_r(y)=\varphi_r({\cal{T}}_{(j)}^{-1}(y)), \
      r=1,2,
\end{equation}
\begin{equation}
  \label{eq:reg-5-4}
      \widetilde{v}(y)=v({\cal{T}}_{(j)}^{-1}(y)), \quad
      \widetilde{\psi}_r(y)=\psi_r({\cal{T}}_{(j)}^{-1}(y)), \
      r=1,2.
\end{equation}
Then, the pair $(\widetilde{\varphi}, \widetilde{w}) \in
H^1(B_\sigma^+, \R^2) \times H^1(B_\sigma^+)$ satisfies
\begin{equation}
  \label{eq:reg-5-5}
      \widetilde{a}_+((\widetilde{\varphi}, \widetilde{w}),
      (\widetilde{\psi}, \widetilde{v}))=
      \widetilde{{F}}_+(\widetilde{\psi}, \widetilde{v}), \ \
      \hbox{for every } (\widetilde{\psi}, \widetilde{v}) \in
      H^1_{\Gamma_\sigma^+}(B_\sigma^+,\R^2)\times H^1_{\Gamma_\sigma^+}(B_\sigma^+),
\end{equation}
where
\begin{multline}
  \label{eq:reg-5-6}
    \widetilde{a}_+((\widetilde{\varphi}, \widetilde{w}),
      (\widetilde{\psi}, \widetilde{v}))= \\
      = \int_{B_\sigma^+} \widetilde{\mathbb
      P}(y) \nabla_y \widetilde{\varphi} \cdot \nabla_y
      \widetilde{\psi}dy +
    \int_{B_\sigma^+} \widetilde{S}(y)(\widetilde{\varphi}+ L^T
    \nabla_y \widetilde{w}) \cdot (\widetilde{\psi} + L^T \nabla_y
    \widetilde{v}) dy,
\end{multline}
\begin{equation}
  \label{eq:reg-5-7}
      \widetilde{F}_+( \widetilde{\psi}, \widetilde{v})=
      \int_{\Gamma_\sigma} (\widetilde{{\cal{Q}}} \widetilde{v} +
      \widetilde{{\cal{M}}}\cdot \widetilde{\psi})ds_y,
\end{equation}
with
\begin{equation}
  \label{eq:reg-5-8}
      (L)_{ks}= L_{ks}= \frac{\partial {\cal {T}}_k}{\partial
      x_s}, \quad k,s=1,2,
\end{equation}
\begin{equation}
  \label{eq:reg-5-9}
      \iota= | \det L |, \quad
      \iota^*=
      \sqrt{ \left (  \frac{\partial {{\cal{T}}^{-1}(y)} }{ \partial y } \right )^T   \frac{\partial {{\cal{T}}^{-1}(y)} }{ \partial y }
      \left |_{y_1,y_2=0} \right.   } ,
\end{equation}
\begin{equation}
  \label{eq:reg-5-10}
      (\widetilde{\mathbb P}(y))_{ilrk}=\widetilde{P}_{ilrk}(y)=
      \sum_{j,s=1}^2 P_{ijrs}(
      {\cal{T}}^{-1}(y))L_{ks}L_{lj} \iota^{-1}, \quad i,l,r,k=1,2,
\end{equation}
\begin{equation}
  \label{eq:reg-5-11}
      \widetilde{S}(y)=S({\cal{T}}^{-1}(y))\iota^{-1},
\end{equation}
\begin{equation}
  \label{eq:reg-6-1}
    \widetilde{{\cal{Q}}}(y)= \overline{Q}({\cal{T}}^{-1}(y))\iota^*,
    \quad
    \widetilde{{\cal{M}}}(y)=\overline{M}({\cal{T}}^{-1}(y))\iota^*.
\end{equation}
Since $L \in C^{0,1}(\Omega_j \cap \Omega, \M^2)$ is nonsingular
and there exist two constants $c_1$, $c_2$, only depending on
$M_0$, such that $ 0 < c_1 \leq \iota, \iota^*  \leq c_2$ in
$\Omega_j$, the fourth order tensor $\widetilde{\mathbb P}$ in
\eqref{eq:reg-5-10} has the following properties:
\begin{itemize}
\item[i)] (major symmetry) for every $2 \times 2$ matrices $A$ and
$B$ we have
\begin{equation}
    \label{eq:reg-Ptilde-sym}
    \widetilde{\mathbb P} A \cdot B= A \cdot
\widetilde{\mathbb P} B;
\end{equation}
\item[ii)] (strong ellipticity) there exists a constant
$\kappa_0$, $\kappa_0 >0$ and $\kappa_0$ only depending on $M_0$
and $\xi_0$, such that for every pair of vectors $a$, $b \in \R^2$
and for every $y \in \overline{B}_\sigma^+$ we have
\begin{equation}
    \label{eq:reg-Ptilde-strell}
    \widetilde{\mathbb P}(a \otimes b) \cdot (a \otimes b) \geq
    \kappa_0 |a|^2|b|^2;
\end{equation}
\item[iii)] (regularity) $\widetilde{\mathbb P} \in
C^{0,1}(\overline{B}_\sigma^+, {\cal L} ({\M}^{2}, {\M}^{2}))$.
\end{itemize}
The matrix $\widetilde{S}$ defined in \eqref{eq:reg-5-11} is
symmetric and there exists a constant $\chi_0$, $\chi_0
>0 $ only depending on $\sigma_0$ and $M_0$, such that for every
vector $a \in \R^2$ and for every $y \in \overline{B}_\sigma^+$ we
have
\begin{equation}
    \label{eq:reg-6-2}
    \widetilde{S}a \cdot a \geq \chi_0 |a|^2.
\end{equation}
Moreover, $\widetilde{S} \in C^{0,1}(B_\sigma^+, {\M}^2)$.

\medskip

We now use the regularity up to the flat boundary of the hemidisk
$B_1^+$ stated in the next theorem, whose proof is postponed in
the Appendix.

\begin{theo}
\label{theo:reg-loc-bound}
Under the above notation and assumptions, let
$(\widetilde{\varphi}, \widetilde{w}) \in H^1(B_\sigma^+,\R^2)
\times H^1(B_\sigma^+)$ defined in \eqref{eq:reg-5-3} be the
solution to \eqref{eq:reg-5-5}. Then $ (\widetilde{\varphi},
\widetilde{w}) \in H^2(B_{ \frac{\sigma}{2}  }^+,\R^2) \times
H^2(B_{ \frac{\sigma}{2} }^+)$ and we have
\begin{equation}
    \label{eq:reg-7-1}
    \|\widetilde{\varphi}\|_{H^2(B_{ \frac{\sigma}{2}  }^+)}+ \|\widetilde{w}\|_{H^2(B_{ \frac{\sigma}{2}  }^+)}
    \leq
    C \left (
    \| \widetilde{{\cal{Q}}}\|_{H^{ \frac{1}{2}  }(\Gamma_\sigma)} + \| \widetilde{{\cal{M}}}\|_{H^{ \frac{1}{2}
    }(\Gamma_\sigma)}+
    \|\widetilde{\varphi}\|_{H^1(B_\sigma^+)}+ \|\widetilde{w}\|_{H^1(B_\sigma^+)}
    \right ),
\end{equation}
where the constant $C>0$ only depends on $M_0$, $\xi_0$,
$\sigma_0$, $\|S\|_{C^{0,1}(\overline{\Omega})}$ and
$\|P\|_{C^{0,1}(\overline{\Omega})}$.
\end{theo}
Recalling that $\Omega = \Omega_0 \cup \left ( \bigcup_{j=1}^N
{\cal T}_{(j)}^{-1}(B_{ \frac{\sigma}{2}}^+) \right )$, the
estimate \eqref{eq:reg-1-1} follows by applying the inverse
mapping ${\cal{T}}_{(j)}^{-1}$ to \eqref{eq:reg-7-1}, $j=1,...,N$,
and by using the interior estimate \eqref{eq:reg-4-2}.

\medskip

\section{Three sphere inequality and strong unique continuation} \label{sec:UC}
In the present section we assume that $\Omega$ is a bounded domain in $\R^2$
of Lipschitz class with  constants $\rho_0$, $M_0$ and we assume that the plate is isotropic with Lam\'{e} parameters
$\lambda,\mu$. We assume that $\lambda,\mu \in C^{0,1}(\overline{\Omega})$ and that, for given positive constants $\alpha_0,\alpha_1,\gamma_0$, they satisfy the following conditions

\begin{equation}
  \label{eq:ser-12}
    \mu(x)\geq \alpha_0, \quad 2\mu(x)+3\lambda(x)\geq \gamma_0,
\end{equation}
and
\begin{equation}
  \label{eq:ser-13}
    \|\lambda\|_{C^{0,1}(\overline{\Omega})}+\|\mu\|_{C^{0,1}(\overline{\Omega})}\leq \alpha_1.
\end{equation}

We assume that the \textit{plate shearing matrix} has the form $SI_2$ where $S\in C^{0,1}(\overline{\Omega})$ is the real valued function defined by
\begin{equation}
  \label{eq:ser-8}
    S=\frac{Eh}{2(1+\nu)},
\end{equation}
where
\begin{equation}
  \label{eq:ser-20}
   E=\frac{\mu(2\mu+3\lambda)}{\mu+\lambda}, \quad \nu=\frac{\lambda}{2(\mu+\lambda)}
\end{equation}
and we assume that \textit{plate bending tensor} $\mathbb P$ has the following form
\begin{equation}
  \label{eq:ser-10}
    {\mathbb P} A = B\left[(1-\nu)\widehat{A}+\nu tr(A)I_2\right], \quad \hbox{for every } 2 \times 2 \quad \hbox{matrix } A,
\end{equation}
where
\begin{equation}
  \label{eq:ser-11}
    B=\frac{Eh^3}{12(1-\nu^2)}.
\end{equation}
By \eqref{eq:ser-12} and \eqref{eq:ser-13} and noticing that $S=h\mu$, we have that
\begin{equation}
  \label{eq:convex-S-ser}
    h \sigma_0 \leq S, \quad \hbox{in } \Omega,\quad  \|S\|_{C^{0,1}(\overline{\Omega})}\leq h \sigma_1
\end{equation}
and
\begin{equation}
  \label{eq:convex-P-ser}
     \frac{h^3}{12} \xi_0 | \widehat{A} |^2 \leq {\mathbb P}A \cdot A \leq \frac{h^3}{12} \xi_1 | \widehat{A} |^2, \quad \hbox{in } \Omega,
\end{equation}
for every $2\times 2$ matrix $A$, where

\begin{equation}
  \label{eq:constants_dependence}
     \sigma_0 =\alpha_0, \quad \sigma_1 =\alpha_1, \quad \xi_0=\min\{2\alpha_0, \gamma_0\}, \quad
        \xi_1=2\alpha_1.
\end{equation}

\begin{theo}
   \label{theo:three-sphere}

Under the the above hypotheses on $\Omega$, $S$ and $\mathbb P$,
let $(\varphi, w) \in H_{loc}^2(\Omega, \R^2) \times
H_{loc}^2(\Omega)$ be a solution of the system
\begin{center}
\( {\displaystyle \left\{
\begin{array}{lr}
    \mathrm{\divrg}(S(\varphi+\nabla w))=0,
    & \mathrm{in}\ \Omega,
        \vspace{0.25em}\\
    \mathrm{\divrg}({\mathbb P}\nabla \varphi)-S(\varphi+\nabla w)=0, & \mathrm{in}\ \Omega.
          \vspace{0.25em}\\

\end{array}
\right. } \) \vskip -4.5em
\begin{eqnarray}
& & \label{eq:ser-29-10}\\
& & \label{eq:ser-29-11}
\end{eqnarray}
\end{center}
Let $\bar x\in\Omega$ and $R_1>0$ be such that $B_{R_1}(\bar x)\subset \Omega$. Then there exists $\theta\in (0,1)$, $\theta$ depends on $\alpha_0,\alpha_1,\gamma_0, \frac{\rho_0}{h}$ only, such that if $0<R_3<R_2<R_1$ and $\frac{R_3}{R_1}\leq \frac{R_2}{R_1}\leq \theta$ then we have

\begin{equation}
  \label{eq:three-sphere-1}
     \int_{B_{R_2}(\bar x)} \left|V\right|^2\leq C\left(\int_{B_{R_3}(\bar x)} \left|V\right|^2\right)^{\tau}\left(\int_{B_{R_1}(\bar x)} \left|V\right|^2\right)^{1-\tau}
\end{equation}
where
\begin{equation}
  \label{eq:three-sphere-2}
   \left|V\right|^2=|\varphi|^2+\frac{1}{\rho^2_0}|w|^2,
\end{equation}
$\tau\in(0,1)$ depends on $\alpha_0,\alpha_1,\gamma_0, \frac{R_3}{R_1}, \frac{R_2}{R_1}, \frac{\rho_0}{h}$ only and $C$ depends on $\alpha_0,\alpha_1,\gamma_0, \frac{R_2}{R_1}, \frac{\rho_0}{h}$ only. In addition, keeping $R_2, R_1$ fixed, we have

\begin{equation}
  \label{eq:three-sphere-3}
   \tau=\mathcal{O}\left(\left|\log R_3\right|^{-1}\right), \quad \hbox{as }   R_3\rightarrow 0.
\end{equation}

\end{theo}

\begin{proof}
It is not restrictive to assume that $\bar x=0\in \Omega$. In order to prove \eqref{eq:three-sphere-1}, first we introduce an auxiliary unknown which allows us to obtain a new system of equations with the Laplace operator as the principal part, then we obtain \eqref{eq:three-sphere-1} by applying \cite[Theorem 1.1]{LNW2010}.
By \eqref{eq:ser-8} and \eqref{eq:ser-10} we have
\begin{multline}
    \label{eq:ser-29-20}
    \mathrm{\divrg}({\mathbb P}\nabla \varphi)-S(\varphi+\nabla w)=
    \\
    =\frac{h^3}{12}\left[\mathrm{\divrg}\left(\mu\left(\nabla\varphi+\nabla^T\varphi\right)\right)+
    \nabla\left(\frac{2\lambda\mu}{2\mu+\lambda}\mathrm{\divrg}\varphi\right)-\frac{12\mu}{h^2}(\varphi+\nabla w)\right].
\end{multline}

\medskip

Now we denote
\begin{equation}
\label{eq:ser-29-21}
  a=\frac{2\mu+3\lambda}{4(\lambda+\mu)}, \quad b=\frac{4(\lambda+\mu)}{2\mu+\lambda},
\end{equation}

\medskip

\begin{equation*}
  G=\left(\nabla\varphi+\nabla^T\varphi\right)\frac{\nabla\mu}{\mu}-\left[\frac{\nabla\mu}{\mu}+
  \frac{\mu(2\mu+3\lambda)}{2\mu+\lambda}\nabla\left(\frac{1}{\mu}\right)\right]\mathrm{\divrg}\varphi
\end{equation*}
and

\medskip

\begin{equation}
\label{eq:ser-29-25}
  v=b\mathrm{\divrg}\varphi.
\end{equation}

\medskip

By \eqref{eq:ser-29-20} we have
\begin{equation*}
     \mathrm{\divrg}({\mathbb P}\nabla \varphi)-S(\varphi+\nabla w)=\frac{h^3\mu}{12}\left[\Delta\varphi+\nabla(av)+G-\frac{12}{h^2}(\varphi+\nabla w)\right],
\end{equation*}
therefore equation \eqref{eq:ser-29-11} is equivalent to the equation

\begin{equation}
  \label{eq:ser-29-30}
     \Delta\varphi+\nabla(av)+G-\frac{12}{h^2}(\varphi+\nabla w)=0.
\end{equation}
Now, noticing that \eqref{eq:ser-29-21} gives $a+\frac{1}{b}=1$, we have
\begin{multline}
\label{eq:ser-29-40}
\mathrm{\divrg}\left(\Delta\varphi+\nabla(av)\right)=\Delta \left(\frac{v}{b}\right)+\Delta \left(av\right)= \Delta \left(\left(a+\frac{1}{b}\right)v\right)=\Delta v.
\end{multline}
Now we apply the divergence operator to both the sides of \eqref{eq:ser-29-30} and by \eqref{eq:ser-29-40}
we get
\begin{equation}
\label{eq:ser-29-50}
\Delta v+\mathrm{\divrg}G-\frac{12}{h^2}\mathrm{\divrg}(\varphi+\nabla w)=0.
\end{equation}
Finally, observing that by equation \eqref{eq:ser-29-10} we have
\begin{equation*}
     \mathrm{\divrg}(\varphi+\nabla w)=\mathrm{\divrg}\left(\frac{1}{S}S(\varphi+\nabla w)\right)=\nabla\left(\frac{1}{S}\right)\cdot S(\varphi+\nabla w),
\end{equation*}
by \eqref{eq:ser-29-50} we obtain
\begin{equation}
\label{eq:ser-29-60}
\Delta v+\mathrm{\divrg}G-\frac{12}{h^2}\nabla\left(\frac{1}{S}\right)\cdot S(\varphi+\nabla w)=0.
\end{equation}
On the other side by \eqref{eq:ser-29-25} we have
\begin{equation}
\label{eq:ser-29-61}
\mathrm{\divrg}(S(\varphi+\nabla w))=S\Delta w+\frac{S}{b}v+\nabla S\cdot\varphi+\nabla S\cdot \nabla w,
\end{equation}
therefore, by  \eqref{eq:ser-29-61}, \eqref{eq:ser-29-10}, \eqref{eq:ser-8} and \eqref{eq:ser-20}, we have

\begin{equation}
\label{eq:ser-29-70}
\Delta w+\frac{2\mu+\lambda}{4(\lambda+\mu)}v+\frac{\nabla S}{S}\cdot\varphi+\frac{\nabla S}{S}\cdot\nabla w=0.
\end{equation}

\medskip

Now, in order to satisfy the homogeneity of norms we define
\begin{equation*}
 \widetilde{w}=w, \quad \widetilde{\varphi}=\rho_0\varphi, \quad \widetilde{v}=\rho^2_0v
\end{equation*}
and
\begin{equation*}
 \widetilde{G}=\rho_0G=\left(\nabla\widetilde{\varphi}+\nabla^T\widetilde{\varphi}\right)\frac{\nabla\mu}{\mu}-\left[\frac{\nabla\mu}{\mu}+
  \frac{\mu(2\mu+3\lambda)}{2\mu+\lambda}\nabla\left(\frac{1}{\mu}\right)\right]\mathrm{\divrg}\widetilde{\varphi}.
\end{equation*}

By  \eqref{eq:ser-29-30}, \eqref{eq:ser-29-60}, \eqref{eq:ser-29-70}, we have that $\widetilde{w},\widetilde{\varphi},\widetilde{v}$ satisfy the system

\begin{equation}
\label{1-141}
\left\{\begin{array}{ll}
\Delta \widetilde{w}+\frac{2\mu+\lambda}{4\rho^2_0(\lambda+\mu)}\widetilde{v}+\frac{\nabla S}{\rho_0 S}\cdot\widetilde{\varphi}+\frac{\nabla S}{S}\cdot\nabla \widetilde{w}=0, \quad \hbox{in } \Omega,\\[2mm]
\Delta\widetilde{\varphi}+\nabla(\frac{a}{\rho_0}\widetilde{v})+\widetilde{G}-\frac{12}{h^2}(\widetilde{\varphi}+\rho_0\nabla \widetilde{w})=0, \quad \hbox{in } \Omega,\\[2mm]
\Delta \widetilde{v}+\rho_0\mathrm{\divrg}\widetilde{G}-\frac{12}{h^2}\rho_0\nabla\left(\frac{1}{S}\right)\cdot S(\widetilde{\varphi}+\rho_0\nabla \widetilde{w})=0, \quad \hbox{in } \Omega.
\end{array}\right.
\end{equation}

\bigskip

The above system has the same form of system (1.5) of \cite{LNW2010}. As a matter of fact, as soon as we introduce the following notation

\begin{equation*}
u=\left(\widetilde{w},\widetilde{\varphi}\right),
\end{equation*}

\begin{equation*}
P_1(x,\partial)\widetilde{v}
=\left(
\begin{array}{c}
\frac{2\mu+\lambda}{4\rho^2_0(\lambda+\mu)}\widetilde{v}\\
\\
\nabla(\frac{a}{\rho_0}\widetilde{v})
\end{array}%
\right), \quad P_2(x,\partial)u=\left(
\begin{array}{c}
\frac{\nabla S}{\rho_0 S}\cdot\widetilde{\varphi}+\frac{\nabla S}{S}\cdot\nabla \widetilde{w}\\
\\
\widetilde{G}-\frac{12}{h^2}(\widetilde{\varphi}+\rho_0\nabla \widetilde{w}),
\end{array}%
\right)
\end{equation*}

\begin{equation*}
Q_1(x,\partial)\widetilde{v}
=0, \quad Q_2(x,\partial)u=-\frac{12}{h^2}\rho_0\nabla\left(\frac{1}{S}\right)\cdot S(\widetilde{\varphi}+\rho_0\nabla \widetilde{w}),
\end{equation*}
system \eqref{1-141} is equivalent to

\begin{equation}
\label{ser30-20}
\left\{\begin{array}{ll}
\Delta u+P_1(x,\partial)\widetilde{v}+P_2(x,\partial)u=0, \quad \hbox{in } \Omega,\\[2mm]
\Delta \widetilde{v}+Q_1(x,\partial)\widetilde{v}+Q_2(x,\partial)u+\rho_0\mathrm{\divrg}\widetilde{G}=0, \quad \hbox{in } \Omega.
\end{array}\right.
\end{equation}
Notice that, likewise to \cite{LNW2010}, $P_j(x,\partial)$ and $Q_j(x,\partial)$, $j=1,2$, are first order operators with $L^{\infty}$ coefficients. In addition, although $\widetilde{G}$ is slightly different from the term $G$ of \cite{LNW2010}, the proof of Theorem 1.1 (after the scaling $x\rightarrow R_1 x$)  of such a paper can be used step by step to derive \eqref{eq:three-sphere-1}.
\end{proof}

\begin{cor}
\label{sucp}
Assume that $S$, $\mathbb P$ and $\Omega$ satisfy the same hypotheses of \ref{theo:three-sphere}, let $x_0\in\Omega$ and let  $(\varphi, w) \in H_{loc}^2(\Omega, \R^2) \times H_{loc}^2(\Omega)$ be a solution of the system \eqref{eq:ser-29-10}-\eqref{eq:ser-29-11} such that

\begin{equation}
\label{ser30-30}
\|\varphi\|_{L^2(B_r(\bar x))}+\frac{1}{\rho_0}\|w\|_{L^2(B_r(\bar x))}=\mathcal{O}\left(r^N\right), \quad \hbox{as } r\rightarrow 0, \quad \forall N\in\mathbb{N}
\end{equation}
then $\varphi\equiv 0$, $w\equiv 0$ in $\Omega$
\end{cor}
\begin{proof}
It is standard consequence of the inequality \eqref{eq:three-sphere-1} and of the connectedness of $\Omega$. For more details see \cite[Corollary 6.4]{MRV07}.
\end{proof}

\section{Appendix}
\label{sec:appendix}

In this appendix we sketch a proof of Theorem
\ref{theo:reg-loc-bound}.

Without loss of generality, we can assume $\sigma=1$. Our proof
consists of two main steps. As first step, we estimate the partial
derivatives $ \frac{\partial }{\partial y_1} \nabla
\widetilde{\varphi}$, $ \frac{\partial }{\partial y_1} \nabla
\widetilde{w}$ along the direction $e_1$ parallel to the flat
boundary $\Gamma_1$ of $B_1^+$. The second step will concern with
the estimate of the partial derivatives $ \frac{\partial
}{\partial y_2} \nabla \widetilde{\varphi}$, $ \frac{\partial
}{\partial y_2} \nabla \widetilde{w}$ along the direction
orthogonal to the flat boundary $\Gamma_1$.

\medskip

\textit{First step.} (Estimate of the tangential derivatives)

Let $\vartheta \in C^\infty_0 (\R^2)$ be a function such that $0
\leq \vartheta(y) \leq 1$ in $\R^2$, $\vartheta \equiv 1$ in
$B_\rho$, $\vartheta \equiv 0$ in $\R^2 \setminus B_{\eta}$, and
$|\nabla^k \vartheta| \leq C$, $k=1,2$, where $\rho= \frac{3}{4}$,
$\eta=\frac{7}{8}$ and $C>0$ is an absolute constant.

For every functions $\widetilde{\psi} \in H_{\Gamma_1^+}^1(B_1^+,
\R^2)$, $\widetilde{v} \in H_{\Gamma_1^+}^1(B_1^+)$, we still
denote by $\widetilde{\psi} \in H^1(\R_+^2, \R^2)$, $\widetilde{v}
\in H^1(\R_+^2)$ their corresponding extensions to $\R_+^2$
obtained by taking $ \widetilde{\psi}=0$, $\widetilde{v}=0$ in
$\R_+^2 \setminus B_1^+$.

Given a real number $s \in \R \setminus \{0\}$, the difference
operator in direction $y_1$ of any function $f$ is defined as
\begin{equation}
    \label{eq:reg-9-1}
    (\tau_{1,s} f)(y) = \frac{f(y+se_1)- f(y)}{s}.
\end{equation}
In the sequel we shall assume $|s| \leq \frac{1}{16}$. We note
that if $\widetilde{\varphi} \in H^1(B_1^+,\R^2)$, $\widetilde{w}
\in H^1(B_1^+)$, then $\tau_{1,s}(\vartheta \widetilde{\varphi})
\in H^1_{\Gamma_1^+}(B_1^+,\R^2)$ and  $\tau_{1,s}(\vartheta
\widetilde{w}) \in H^1_{\Gamma_1^+}(B_1^+)$.

We start by evaluating the bilinear form $
\widetilde{a}_+((\cdot,\cdot), (\cdot,\cdot))$ defined in
\eqref{eq:reg-5-6} with $\widetilde{\varphi}$, $\widetilde{w}$
replaced by $\tau_{1,s}(\vartheta \widetilde{\varphi})$,
$\tau_{1,s}(\vartheta \widetilde{w})$, respectively. Next, we
elaborate the expression of $ \widetilde{a}_+((\cdot,\cdot),
(\cdot,\cdot))$ and, by integration by parts, we move the
difference operator in direction $y_1$ {}from the functions
$\vartheta \widetilde{\varphi}$, $\vartheta \widetilde{w}$ to the
functions $\widetilde{\psi}$, $\widetilde{v}$. After these
calculations, we can write
\begin{equation}
    \label{eq:reg-9-2}
    \widetilde{a}_+( (\tau_{1,s}(\vartheta \widetilde{\varphi}), \tau_{1,s}(\vartheta
    \widetilde{w})), (\widetilde{\psi}, \widetilde{v})) =
    - \widetilde{a}_+( (\widetilde{\varphi}, \widetilde{w}),
    (\vartheta \tau_{1,-s}\widetilde{\psi}, \vartheta \tau_{1,-s}
    \widetilde{v})) + \widetilde{r},
\end{equation}
where the remainder $\widetilde{r}$ can be estimated as follows
\begin{equation}
    \label{eq:reg-9-3}
    |\widetilde{r}| \leq
    C
    \left (
    \|\widetilde{\varphi}\|_{H^1(B_1^+)} + \|\widetilde{w}\|_{H^1(B_1^+)}
    \right )
    \left (
    \|\nabla \widetilde{\psi}\|_{L^2(B_1^+)} + \|\nabla \widetilde{v}\|_{L^2(B_1^+)}
    \right ),
\end{equation}
where the constant $C>0$ depends on $M_0$,
$\|P\|_{C^{0,1}(\overline{\Omega})}$ and
$\|S\|_{C^{0,1}(\overline{\Omega})}$ only. It should be noticed
that a constructive Poincar\'{e} inequality for functions
belonging to $H^1(B_1^+)$ and vanishing on the portion
$\Gamma_1^+$ of the boundary of $B_1^+$ has been used in obtaining
\eqref{eq:reg-9-3}, see, for example, \cite{A-M-R02}.

Since $\widetilde{\psi} \in H^1_{\Gamma_1^+}(B_1^+,\R^2)$,
$\widetilde{v} \in H^1_{\Gamma_1^+}(B_1^+)$, the functions
$\vartheta \tau_{1,-s}\widetilde{\psi}$, $\vartheta
\tau_{1,-s}\widetilde{v}$ are test functions in the weak
formulation \eqref{eq:reg-5-5}, so that the opposite of the first
term on the right hand side of \eqref{eq:reg-9-2} can be written
as
\begin{equation}
    \label{eq:reg-10-1}
    \widetilde{a}_+( (\widetilde{\varphi}, \widetilde{w}),
    (\vartheta \tau_{1,-s}\widetilde{\psi}, \vartheta \tau_{1,-s}
    \widetilde{v}))=\widetilde{F}_+ ( \vartheta \tau_{1,-s}
    \widetilde{\psi}, \vartheta \tau_{1,-s}
    \widetilde{v})
\end{equation}
and, by using trace inequalities, we have
\begin{equation}
    \label{eq:reg-10-2}
    |\widetilde{F}_+ ( \vartheta \tau_{1,-s}
    \widetilde{\psi}, \vartheta \tau_{1,-s}
    \widetilde{v})| \leq
    C
    \left (
    \| \widetilde{{\cal{Q}}}\|_{H^{ \frac{1}{2}  }(\Gamma_1)} \cdot \|\nabla \widetilde{v} \|_{L^2(B_1^+)}+
    \| \widetilde{{\cal{M}}}\|_{H^{ \frac{1}{2}   }(\Gamma_1)} \cdot \|\nabla \widetilde{\psi} \|_{L^2(B_1^+)}
    \right ),
\end{equation}
where $C>0$ only depends on $M_0$. By
\eqref{eq:reg-9-2}--\eqref{eq:reg-10-2} we have
\begin{multline}
    \label{eq:reg-10-3}
    \widetilde{a}_+( (\tau_{1,s}(\vartheta \widetilde{\varphi}), \tau_{1,s}(\vartheta
    \widetilde{w})), (\widetilde{\psi}, \widetilde{v})) \leq
    \\
    \leq
    C
    \left (
    \| \widetilde{{\cal{Q}}}\|_{H^{ \frac{1}{2}  }(\Gamma_1)} +
    \| \widetilde{{\cal{M}}}\|_{H^{ \frac{1}{2}  }(\Gamma_1)} +
    \| \widetilde{\varphi}\|_{H^1(B_1^+)} + \| \widetilde{w}\|_{H^1(B_1^+)}
    \right )
    \cdot
    \\
    \cdot
    \left (
    \|\nabla \widetilde{v} \|_{L^2(B_1^+)}+
    \|\nabla \widetilde{\psi} \|_{L^2(B_1^+)}
    \right ),
\end{multline}
for every $(\widetilde{\psi}, \widetilde{v}) \in
H^1_{\Gamma_1^+}(B_1^+, \R^2) \times H^1_{\Gamma_1^+}(B_1^+)$,
where $C>0$ only depends on $M_0$, $\|\mathbb
P\|_{C^{0,1}(\overline{\Omega})}$ and
$\|S\|_{C^{0,1}(\overline{\Omega})}$.

We choose in \eqref{eq:reg-10-3} the test functions
\begin{equation}
    \label{eq:reg-11-1}
    \widetilde{\psi}=\tau_{1,s}(\vartheta \widetilde{\varphi}),
    \quad \widetilde{v}= \tau_{1,s}(\vartheta \widetilde{w}).
\end{equation}
The next step consists in estimating {}from below the quadratic
form $ \widetilde{a}_+((\cdot,\cdot), (\cdot,\cdot))$. To perform
this estimate, we write
\begin{equation}
    \label{eq:reg-11-2}
    \widetilde{a}_+( (\tau_{1,s}(\vartheta \widetilde{\varphi}), \tau_{1,s}(\vartheta
    \widetilde{w})), (\tau_{1,s}(\vartheta \widetilde{\varphi}), \tau_{1,s}(\vartheta
    \widetilde{w}))= \widetilde{a}_+^{\widetilde{\mathbb P}} (\tau_{1,s}(\vartheta
    \widetilde{\varphi})) +
    \widetilde{a}_+^{\widetilde{S}} (\tau_{1,s}(\vartheta
    \widetilde{\varphi}), \tau_{1,s}(\vartheta
    \widetilde{w}) ),
\end{equation}
where
\begin{equation}
    \label{eq:reg-11-3}
    \widetilde{a}_+^{\widetilde{\mathbb P}} (\tau_{1,s}(\vartheta
    \widetilde{\varphi}))= \int_{B_1^+} \widetilde{\mathbb P}
    \nabla (\tau_{1,s}(\vartheta
    \widetilde{\varphi})) \cdot \nabla (\tau_{1,s}(\vartheta
    \widetilde{\varphi})),
\end{equation}
\begin{multline}
    \label{eq:reg-11-4}
    \widetilde{a}_+^{\widetilde{S}} (\tau_{1,s}(\vartheta
    \widetilde{\varphi}), \tau_{1,s}(\vartheta
    \widetilde{w}) )= \\
    =\int_{B_1^+}  \widetilde{S}
    \left ( \tau_{1,s}(\vartheta
    \widetilde{\varphi})+ L^T \nabla (\tau_{1,s}(\vartheta
    \widetilde{w})) \right ) \cdot \left ( \tau_{1,s}(\vartheta
    \widetilde{\varphi})+ L^T \nabla (\tau_{1,s}(\vartheta
    \widetilde{w})) \right ).
\end{multline}
By \eqref{eq:reg-6-2}, the matrix $\widetilde{S}$ is definite
positive, and then $\widetilde{a}_+^{\widetilde{S}}(\cdot,\cdot)$
can be easily estimated {}from below as follows
\begin{equation}
    \label{eq:reg-11-5}
    \widetilde{a}_+^{\widetilde{S}} (\tau_{1,s}(\vartheta
    \widetilde{\varphi}), \tau_{1,s}(\vartheta
    \widetilde{w}) )\geq
    C \int_{B_1^+}  | \tau_{1,s}(\vartheta
    \widetilde{\varphi})+ L^T \nabla (\tau_{1,s}(\vartheta
    \widetilde{w})) |^2,
\end{equation}
where $C>0$ only depends on $M_0$ and $\sigma_0$.

The fourth order tensor $ \widetilde{\mathbb P}$ neither has the
minor symmetries nor is strongly convex. Then, in order to
estimate {}from below $\widetilde{a}_+^{\widetilde{\mathbb P}}
(\tau_{1,s}(\vartheta \widetilde{\varphi}))$, we found convenient
apply the inverse transformation ${\cal{T}}_{(j)}^{-1}$ (see
\eqref{eq:reg-5-2}) and use the strong convexity of the tensor
$\mathbb P$. To simplify the notation, let $\widetilde{f} \equiv
\tau_{1,s}(\vartheta \widetilde{\varphi}) \in H_{\Gamma_1^+}^1
(B_1^+, \R^2)$. We have
\begin{equation}
    \label{eq:reg-12-1}
    \widetilde{a}_+^{\widetilde{\mathbb P}}(\widetilde{f})= \int_{B_1^+}
    \widetilde{\mathbb P}(y) \nabla_y \widetilde{f} \cdot \nabla_y
    \widetilde{f}dy = \int_{\Omega_j \cap \Omega} \mathbb P(x)
    \nabla_x f \cdot \nabla_x f dx \geq C \int_{\Omega_j \cap
    \Omega}| \widehat{\nabla}_x f|^2 dx,
\end{equation}
where $f(x)= \widetilde{f}( {\cal{T}}_{(j)}(x))$ and $C>0$ is a
constant only depending on $\xi_0$. By Korn's inequality on
$H_{\Gamma_1^+}(B_1^+,\R^2)$ (see, for example, Theorem 5.7 in
\cite{A-M-R08}) and by the change of variables
$y={\cal{T}}_{(j)}(x)$, we have
\begin{equation}
    \label{eq:reg-12-2}
    \int_{\Omega_j \cap \Omega}| \widehat{\nabla}_x f|^2 dx \geq C
    \int_{\Omega_j \cap \Omega}| \nabla_x f|^2 dx =
    \int_{B_1^+} |\nabla_y \widetilde{f} L|^2 \iota^{-1}dy \geq C'
    \int_{B_1^+} |\nabla_y \widetilde{f}|^2 dy,
\end{equation}
where $C'>0$ only depends on $M_0$, and in the last step we have
taken into account that the matrix $L$ is nonsingular. Then, by
\eqref{eq:reg-12-1} and \eqref{eq:reg-12-2}, we have
\begin{equation}
    \label{eq:reg-13-1}
    \widetilde{a}_+^{\widetilde{\mathbb P}}(\tau_{1,s}(\vartheta
    \widetilde{\varphi})) \geq C \int_{B_1^+} | \nabla (\tau_{1,s}(\vartheta
    \widetilde{\varphi}))|^2,
\end{equation}
where $C>0$ only depends on $M_0$ and $\xi_0$. Now, by inserting
the estimates \eqref{eq:reg-11-5} and \eqref{eq:reg-13-1} in
\eqref{eq:reg-10-3}, with $\widetilde{\psi}$, $\widetilde{v}$ as
in \eqref{eq:reg-11-1}, and by Poincar\'{e}'s inequality in
$H^1_{\Gamma_1^+}(B_1^+)$, we have
\begin{multline}
    \label{eq:reg-14-1}
    \| \nabla(\tau_{1,s}(\vartheta
    \widetilde{\varphi}))\|_{L^2(B_1^+)} + \|\tau_{1,s}(\vartheta
    \widetilde{\varphi}) + L^T \nabla (\tau_{1,s}(\vartheta
    \widetilde{w}))\|_{L^2(B_1^+)} \leq
    \\
    \leq
    C\left (
    \| \widetilde{{\cal{Q}}}\|_{H^{ \frac{1}{2}  }(\Gamma_1)} +
    \| \widetilde{{\cal{M}}}\|_{H^{ \frac{1}{2}  }(\Gamma_1)} +
    \| \widetilde{\varphi}\|_{H^1(B_1^+)} + \| \widetilde{w}\|_{H^1(B_1^+)}
    \right )
\end{multline}
where $C>0$ only depends on $M_0$, $\xi_0$, $\sigma_0$, $\|\mathbb
P\|_{C^{0,1}(\overline{\Omega})}$ and
$\|S\|_{C^{0,1}(\overline{\Omega})}$. Taking the limit as $s
\rightarrow 0$ and by the definition of the function $\vartheta$,
we have
\begin{multline}
    \label{eq:reg-14-2}
    \left \| \frac{\partial }{\partial y_1} \nabla \widetilde{\varphi}  \right  \|_{L^2(B_{\rho}^+)}
    + \left \| \frac{\partial \widetilde{\varphi} }{\partial y_1} + L^T \frac{\partial }{\partial y_1} \nabla \widetilde{w}  \right \|_{L^2(B_{\rho}^+)} \leq
    \\
    \leq
    C \left (
    \| \widetilde{{\cal{Q}}}  \|_{H^{ \frac{1}{2}  }(\Gamma_1)} +
     \| \widetilde{{\cal{M}}}   \|_{H^{ \frac{1}{2}  }(\Gamma_1)} +
    \| \widetilde{\varphi}\|_{H^1(B_1^+)} + \| \widetilde{w}\|_{H^1(B_1^+)}
    \right )
\end{multline}
where $C>0$ only depends on $M_0$, $\xi_0$, $\sigma_0$, $\|\mathbb
P\|_{C^{0,1}(\overline{\Omega})}$ and
$\|S\|_{C^{0,1}(\overline{\Omega})}$. Therefore, the tangential
derivatives $\frac{\partial }{\partial y_1} \nabla
\widetilde{\varphi} $, $\frac{\partial }{\partial y_1} \nabla
\widetilde{w}$ exist and belong to $L^2(B_{\rho}^+)$.

\medskip

\textit{Second step.} (Estimate of the normal derivatives)

To obtain an analogous estimate of the normal derivatives
$\frac{\partial }{\partial y_2} \nabla \widetilde{\varphi} $,
$\frac{\partial }{\partial y_2} \nabla \widetilde{w}$ we need to
prove the following two facts:
\begin{equation}
    \label{eq:reg-15-1}
    \left | \int_{B^+_\rho} \frac{\partial \widetilde{\varphi}_r}{\partial
    y_2}\frac{\partial \widetilde{\psi}}{\partial y_2} \right | \leq C \|
    \widetilde{\psi}\|_{L^2(B^+_\rho)}, \quad \hbox{for every }
    \widetilde{\psi} \in C^\infty_0(B^+_\rho), \ r=1,2,
\end{equation}
\begin{equation}
    \label{eq:reg-15-2}
    \left | \int_{B^+_\rho} \frac{\partial \widetilde{w}}{\partial
    y_2}\frac{\partial \widetilde{v}}{\partial y_2} \right | \leq C \|
    \widetilde{v}\|_{L^2(B^+_\rho)}, \quad \hbox{for every }
    \widetilde{v} \in C^\infty_0(B^+_\rho),
\end{equation}
for some constant $C>0$ depending only on the data. Since
\begin{equation}
    \label{eq:reg-15-3}
    \widetilde{a}_+((\widetilde{\varphi}, \widetilde{w}),
    (\widetilde{\psi}, \widetilde{v}))= 0, \quad \hbox{for every }
    (\widetilde{\psi}, \widetilde{v}) \in  C^\infty_0(B^+_\rho,
    \R^2)\times C^\infty_0(B^+_\rho),
\end{equation}
by integration by parts we have
\begin{multline}
    \label{eq:reg-15-4}
    \int_{B_\rho^+} {\cal{P}}_{ir} \widetilde{\varphi}_{r,2}
    \widetilde{\psi}_{i,2} + \int_{B_\rho^+} {\cal{S}}_{22}
    \widetilde{w},_2 \widetilde{v},_2 = \int_{B_\rho^+}
    \sum^2_{\overset{i,j,r,s=1}{(j,s)\neq (2,2)}}
    (\widetilde{P}_{ijrs}
    \widetilde{\varphi}_{r,s}),_j \widetilde{\psi}_i -
    \\
    - \int_{B_\rho^+}
    \left (
    \widetilde{S} \widetilde{\varphi} \cdot \widetilde{\psi} -
    (\widetilde{S}_{ij} \widetilde{\varphi}_j (L^T)_{ik}),_k \widetilde{v}
    +
    \widetilde{S}(L^T \nabla \widetilde{w}) \cdot
    \widetilde{\psi} - \sum^2_{\overset{i,j=1}{(i,j)\neq (2,2)}}
    ((L\widetilde{S}L^T)_{ij} \widetilde{w},_j),_i \widetilde{v}
    \right ),
\end{multline}
for every $(\widetilde{\psi}, \widetilde{v}) \in
C^\infty_0(B^+_\rho,\R^2)\times C^\infty_0(B^+_\rho)$, where
\begin{equation}
    \label{eq:reg-16-1}
    {\cal{P}}_{ir} = \widetilde{P}_{i2r2}, \  i,r=1,2, \quad
    {\cal{S}}_{22}=(L\widetilde{S}L^T)_{22}.
\end{equation}
By the properties
\eqref{eq:reg-Ptilde-sym}-\eqref{eq:reg-Ptilde-strell} of
$\widetilde{\mathbb {P}}$ and the definite positiveness of
$\widetilde{S}$ (see \eqref{eq:reg-6-2}), the matrix
$({\cal{P}}_{ir})_{i,r=1,2}$ is symmetric and definite positive
and ${\cal{S}}_{22}>0$.

Let $\widetilde{v}=0$ in \eqref{eq:reg-15-4}. Then, by using
estimate \eqref{eq:reg-14-2} we have
\begin{multline}
    \label{eq:reg-16-2}
    \left | \int_{B_\rho^+} {\cal{P}}_{ir} \widetilde{\varphi}_{r,2}
    \widetilde{\psi}_{i,2} \right |
    \leq
    \\
    \leq
    C
    \left (
    \| \widetilde{{\cal{Q}}}\|_{H^{ \frac{1}{2}  }(\Gamma_1)} +
    \| \widetilde{{\cal{M}}}\|_{H^{ \frac{1}{2}  }(\Gamma_1)} +
    \| \widetilde{\varphi}\|_{H^1(B_1^+)} + \| \widetilde{w}\|_{H^1(B_1^+)}
    \right ) \| \widetilde{\psi}\|_{L^2(B_\rho^+)},
\end{multline}
for every $\widetilde{\psi} \in C^\infty_0(B_\rho^+)$, where the
constant $C>0$ only depends on $M_0$, $\xi_0$, $\sigma_0$,
$\|\mathbb P\|_{C^{0,1}(\overline{\Omega})}$ and
$\|S\|_{C^{0,1}(\overline{\Omega})}$. This inequality implies the
existence in $L^2(B_\rho^+)$ of the derivative $ \frac{\partial
}{\partial y_2} \left ( \sum_{r=1}^2 {\cal{P}}_{ir}
\widetilde{\varphi}_{r,2} \right )$, $i=1,2$. Then, it is easy to
see that this condition implies $ \frac{\partial^2
\widetilde{\varphi}_r}{\partial y_2^2} \in L^2(B_\rho^+)$,
$r=1,2$.

Similarly, choosing $\widetilde{\psi}=0$ in \eqref{eq:reg-15-4} we
have
\begin{multline}
    \label{eq:reg-16-3}
    \left | \int_{B_\rho^+} {\cal{S}}_{22}
    \widetilde{w},_2 \widetilde{v},_2 \right |
    \leq
    \\
    \leq
    C
    \left (
    \| \widetilde{{\cal{Q}}}\|_{H^{ \frac{1}{2}  }(\Gamma_1)} +
    \| \widetilde{{\cal{M}}}\|_{H^{ \frac{1}{2}  }(\Gamma_1)} +
    \| \widetilde{\varphi}\|_{H^1(B_1^+)} + \| \widetilde{w}\|_{H^1(B_1^+)}
    \right ) \| \widetilde{v}\|_{L^2(B_\rho^+)},
\end{multline}
for every $\widetilde{v} \in C^\infty_0(B^+_\rho)$, where the
constant $C>0$ only depends on $M_0$, $\xi_0$, $\sigma_0$,
$\|\mathbb P\|_{C^{0,1}(\overline{\Omega})}$ and
$\|S\|_{C^{0,1}(\overline{\Omega})}$. As before, this condition
implies the existence in $L^2(B_\rho^+)$ of $ \frac{\partial^2
\widetilde{w}}{\partial y_2^2}$.

Finally, {}from \eqref{eq:reg-16-2} and \eqref{eq:reg-16-3}, the
$L^2$-norm of $ \frac{\partial^2 \widetilde{\varphi}_r}{\partial
y_2^2}$, $r=1,2$, and $ \frac{\partial^2 \widetilde{w}}{\partial
y_2^2}$ can be estimated in terms of known quantities, and the
proof of Theorem \ref{theo:reg-loc-bound} is complete.

\bigskip
\bibliographystyle{plain}

\begin{thebibliography}{99}


\bibitem[Ag65]{Ag65}
S. Agmon, \textit{Lectures on Elliptic Boundary Value Problems},
Mathematical Studies, vol. 2. D. Van Nostrand Co., Inc.,
Princeton, N.J.-Toronto-London, 1965.


\bibitem[A-M-R02]{A-M-R02} G. Alessandrini, A. Morassi, E. Rosset,
Detecting cavities by electrostatic boundary measurements, Inverse
Problems 18 (2002) 1333--1353.


\bibitem[A-M-R08]{A-M-R08} G. Alessandrini, A. Morassi, E. Rosset,
The linear constraints in Poincar\'{e} and Korn type inequalities,
Forum Mathematicum 20 (2008) 557--569.


\bibitem[Ca80]{Ca80} S. Campanato, \textit{Sistemi ellittici in forma divergenza. Regolarit\`{a}
all'interno.} Quaderni della Scuola Normale Superiore, Pisa, 1980.

\bibitem[DL76]{DL76} G. Duvaut, J.L. Lions, \textit{Inequalities in Mechanics
and Physics}, Springer-Verlag, 1976.

\bibitem[F47]{F47}
K.O. Friedrichs, On the boundary value problems of the theory of
elasticity and Korn's inequality, Annals of Mathematics 48 (1947)
441--471.

\bibitem[G62]{G62}
J. Gobert, Une in\'egalit\'e fondamentale de la th\'eorie de
l'\'elasticit\'e, Bulletin de la Soci\'{e}t\'{e} Royale des
Sciences de Li\`{e}ge 31 (1962) 182--191.

\bibitem[LNW2010]{LNW2010}
C.L. Lin, G. Nakamura, J.N. Wang, Optimal three-Ball inequalities
and quantitative uniqueness for the Lam\'{e} system with Lipschitz
coefficients, Duke Mathematical Journal 155(1) (2010) 189--204.

\bibitem[MRV07]{MRV07}
A. Morassi, E. Rosset, S. Vessella, Size estimates for inclusions
in an elastic plate by boundary measurements Indiana University
Mathematical Journal 56 (2007) 2325�-84

\bibitem[Min51]{Min51}
R.D. Mindlin, Influence of rotatory inertia and shear on flexural
motions of isotropic elastic plates, Journal of Applied Mechanics
18 (1951) 31-�38.


\bibitem[P-PPG-T07]{P-PPG-T-07}
R. Paroni, P. Podio-Guidugli, G. Tomassetti, A justification of
the Reissner-Mindlin plate theory through variational convergence,
Analysis and Applications 5(2) (2007) 165-182.


\bibitem[Rei45]{Rei45}
E. Reissner, The effect of transverse shear deformation on the
bending of elastic plates, Journal of Applied Mechanics 12 (1945)
A69�-A77.


\bibitem[T66]{Truesdell}
C. Truesdell, \textit{The Elements of Continuum Mechanics},
Springer-Verlag, Berlin, 1966.



\end{thebibliography}

\end{document}